\documentclass[12pt]{amsart}
\usepackage{amsmath}
\usepackage{amsthm}
\usepackage{mathtools}
\usepackage{bbm} 
\usepackage{amsfonts,amssymb}
\usepackage{fancyhdr} 
\usepackage{amsfonts,amssymb}  
\usepackage{mathrsfs}  
\usepackage{graphicx}
\usepackage[all]{xy}
\usepackage{color}

\usepackage{float} 
\usepackage{amsthm}

\newtheorem{thm}{Theorem}[section]
\newtheorem{defi}[thm]{Definition}
\newtheorem{lem}[thm]{Lemma}
\newtheorem{cor}[thm]{Corollary}

\newtheorem{ex}[thm]{Example}
\newtheorem{rmk}[thm]{Remark}

\usepackage{geometry} 
\usepackage{times}
\usepackage{hyperref}
\def\Aut{{\rm Aut}}
\def\Hom{{\rm Hom}}
\def\ord{{\rm ord}}
\def\min{{\rm min}}
\def\max{{\rm max}}
\def\inf{{\rm inf}}
\def\sup{{\rm sup}}
\def\lim{{\rm lim}}
\def\limsup{{\rm lim\,sup}}
\def\liminf{{\rm lim\,inf}}
\def\dif{{\rm d}}
\def\interior{{\rm int}}
\def\Supp{{\rm Supp}}
\def\gr{{\rm gr}}
\def\pr{{\rm pr}}
\def\wt{{\rm wt}}
\def\nv{\mathfrak{v}}

\def\Proj{{\rm Proj}}
\def\Spec{{\rm Spec}}
\def\Diff{{\rm Diff}}

\def\triv{{\rm triv}}
\def\log{{\rm log}}
\def\vol{{\rm vol}}
\def\Val{{\rm Val}}

\def\lct{{\rm lct}}
\def\DH{{\rm DH}}
\def\LE{{\rm LE}}

\def\Fut{{\rm Fut}}

\def\Cone{{\rm Cone}}

\def\loc{{\rm loc}}

\def\BP{\mathbf{P}}
\def\BQ{\mathbf{Q}}
\def\BO{\mathbf{O}}

\def\Bv{\mathbf{v}}

\def\BH{\mathbf{H}}
\def\QGamma{\rm{Q}\Gamma}
\def\PGamma{\rm{P}\Gamma}

\newcommand{\IA}{{\mathbb A}}

\newcommand{\IC}{{\mathbb C}}

\newcommand{\IG}{{\mathbb G}}

\newcommand{\Ik}{{\mathbbm k}}

\newcommand{\IN}{{\mathbb N}}
 
\newcommand{\IP}{{\mathbb P}} 
\newcommand{\IQ}{{\mathbb Q}} 
\newcommand{\IR}{{\mathbb R}}

\newcommand{\IT}{{\mathbb T}}

\newcommand{\IZ}{{\mathbb Z}}

\newcommand{\CC}{{\mathcal C}}

\newcommand{\CF}{{\mathcal F}}
\newcommand{\CG}{{\mathcal G}}

\newcommand{\CL}{{\mathcal L}}

\newcommand{\CO}{{\mathcal O}}

\newcommand{\CV}{{\mathcal V}}

\newcommand{\CY}{{\mathcal Y}}

\newcommand{\fm}{\mathfrak{m}}

\newcommand{\seq}{\subseteq}
\newcommand{\la}{\langle}
\newcommand{\ra}{\rangle}

\newcommand{\bu}{\bullet}
\newcommand{\lam}{\lambda}
\newcommand{\D}{\Delta}

\title{K\"ahler-Ricci solitons on Fano threefolds with non-trivial moduli}

\author{Minghao Miao}
\address{Department of Mathematics, Nanjing University, Nanjing 210093, China}
\curraddr{}
\email{minghao.miao@smail.nju.edu.cn}
\thanks{}
\date{}
\dedicatory{}

\author{Linsheng Wang}
\address{Shanghai Center for Mathematical Sciences, Fudan University, Shanghai, 200438, China}
\curraddr{}
\email{linsheng\_wang@fudan.edu.cn}
\thanks{}
\date{}
\dedicatory{}

\begin{document}

\begin{abstract}
We find Fano threefolds $X$ admitting K\"ahler-Ricci solitons (KRS) with non-trivial moduli, which are $\mathbb{T}$-varieties of complexity two. More precisely, we show that the weighted K-stability of $(X,\xi_0)$ (where $\xi_0$ is the soliton candidate) is equivalent to certain GIT-stability. 
In particular, this provides the first examples of strictly weighted K-semistable Fano varieties. 

On the other hand, we generalize Koiso's theorem to the log Fano setting. Indeed, we show that the K-stability of a log Fano pair $(V,\Delta_V)$ is equivalent to the weighted K-stability of a cone $(Y, \Delta_Y, \xi_0)$ over it. This also leads to new examples of KRS Fano varieties with non-trivial moduli and small automorphism groups. 

To achieve these, we establish the weighted Abban-Zhuang estimate generalizing the work of \cite{AZ22}, which gives a lower bound of the weighted stability threshold $\delta^g_{\mathbb{T}}(X,\Delta)$. This is an effective way to check the weighted K-semistablity of a log Fano triple $(X,\Delta,\xi_0)$. This estimate is also useful in testing (weighted) K-polystability based on the work of \cite{BLXZ23}. 
\end{abstract}

%French abstract
%Nous identifions des variétés de Fano tridimensionnelles \(X\) admettant des solitons de Kähler-Ricci (KRS) avec des moduli non triviaux, qui sont des variétés \(\mathbb{T}\)-variétés de complexité deux. Plus précisément, nous démontrons que la K-stabilité pondérée de \( (X, \xi_0) \) (où \( \xi_0 \) est le candidat soliton) est équivalente à une certaine stabilité GIT. En particulier, cela fournit les premiers exemples de variétés de Fano strictement pondérées K-semistables.

%D'autre part, nous généralisons le théorème de Koiso au cadre log-Fano. En effet, nous montrons que la K-stabilité d'une paire log-Fano \( (V, \Delta_V) \) est équivalente à la K-stabilité d'un cône \( (Y, \Delta_Y, \xi_0) \) au-dessus de celui-ci. Cela conduit également à de nouveaux exemples de variétés de Fano KRS avec des moduli non triviaux et de petits groupes d'automorphismes.

%Pour y parvenir, nous établissons l'estimation pondérée d'Abban-Zhuang généralisant le travail de \cite{AZ22}, qui donne une borne inférieure du seuil de stabilité pondérée \( \delta^g_{\mathbb{T}}(X, \Delta) \). C'est une méthode efficace pour vérifier la K-semistabilité pondérée d'un triple log-Fano \( (X, \Delta, \xi_0) \). De manière surprenante, une telle estimation est également utile pour tester la (pondérée) K-polystabilité, basée sur le travail de \cite{BLXZ23}.

\maketitle

%\tableofcontents

\section{Introduction}
The existence of canonical metrics on a given manifold is a fundamental problem in K\"ahler geometry. The so-called Yau-Tian-Donaldson (YTD) conjecture predicts that the existence of canonical metrics is equivalent to certain algebro-geometric stability conditions. The concept of K-stability introduced by Tian \cite{Tia97} detects the existence of the K\"ahler-Einstein (KE) metrics on a Fano manifold. It was proved in \cite{Tia15} that a Fano manifold admits a K\"ahler-Einstein metric if and only if it is K-polystable. Other proofs were also given in \cite{CDS15,DS16,CSW18,BBJ18,Li19,Zhang22}.

However, it is in general very hard to verify the K-stability of a given Fano manifold $X$. The first powerful tool in testing K-stability is Tian's criterion \cite{Tia87} via the $\alpha$-invariant, with a notable application that every smooth Fano hypersurface of index one is K-stable \cite{Fuj19}. More recently, \cite{AZ22} introduced an inductive approach to give a lower bound estimate of the stability threshold $\delta(X)$, which proved that every smooth Fano hypersurface of index $r \ge 3$ and dimension $\ge r^3$ is K-stable \cite{AZ23}. The Abban-Zhuang theory is also widely used in the study of the K-stability of Fano threefolds, see \cite{ACC+}. Besides testing K-stability, the Abban-Zhuang theory is also useful in finding the minimizer of the local delta invariant. For example, \cite[Theorem 4.6]{AZ22} found the minimizer of $\delta_p(X)$ for every cubic surface $X$ and every closed point $p\in X$. 

There are many Fano manifolds not admitting KE metric. The K\"ahler-Ricci soliton (KRS) metrics provide a natural generalization of the KE metrics on Fano manifolds, which is closely related to the limiting behavior of the K\"ahler-Ricci flow on Fano manifolds. The uniqueness of K\"ahler-Ricci solitons is established by \cite{TZ00,TZ02}, 
%so K\"ahler-Ricci soliton is also a type of "canonical metric" on Fano manifolds.
There are several results on the existence of K\"ahler-Ricci solitons on Fano varieties with large symmetry, for example, on toric varieties \cite{WZ04, SZ12, BB13}, on spherical varieties \cite{LZZ18, Del20, LLW22} and on $\IT$-varieties of complexity one \cite{IS17, CS18, HHS23}. 
In this paper, we find new examples of Fano threefolds admitting K\"ahler-Ricci solitons, which are $\IT$-varieties of complexity two (Section \ref{Subsection:torus action}). 
\begin{thm}[Theorem \ref{Theorem. stability of 2.28} and \ref{Theorem: stability of 3.14}]
\label{Theorem: Intro №2.28 and No 3.14 soliton}
Every smooth Fano threefold in the families {\rm №2.28} and {\rm №3.14} of Mori and Mukai's list admits a K\"ahler-Ricci soliton. 
\end{thm}

By the celebrated works of \cite{HL23} and \cite{BLXZ23}, a log Fano triple $(X,\D,\xi_0)$ (where $(X,\D)$ is a log Fano pair with a $\IT$-action and $\xi_0\in N(\IT)_\IR$ is a soliton candidate, see Section \ref{Subsection: soliton candidate}) admits a twisted K\"ahler-Ricci soliton if and only if it is weighted K-polystable (proved by \cite{DS16, CSW18} for smooth $X$ and $\D$ = 0).
It suffices to show that $(X,\xi_0)$ is weighted K-polystable, where $X$ is a Fano threefold in the above families in Theorem \ref{Theorem: Intro №2.28 and No 3.14 soliton}, and the holomorphic vector field $\xi_0$ is the soliton candidate of $X$ (see Section \ref{Subsection: soliton candidate}). Since there is only one-dimensional torus action on $X$, all known methods do not apply. 

To overcome the difficulty, we establish the weighted Abban-Zhuang estimate to show that the weighted stability threshold $\delta^g_{\IT}(X,\D)\geq 1$ (see Section \ref{Subsection. Weighted stability thresholds} for the definition of $\delta^g_{\IT}(X,\D)$ and Remark \ref{Remark. choice of g} for the choice of $g$), which proves the weighted K-semistability of a log Fano triple $(X,\D,\xi_0)$. 

\begin{thm}[Theorem \ref{Theorem: weighted AZ}, weighted Abban-Zhuang estmate]
\label{Theorem: Intro weighted Abban-Zhuang}
Let $F$ be a $\IT$-invariant plt-type divisor (Section \ref{Subsection. Valuations}) over $X$, and $Z \subseteq X$ be a $\IT$-invariant subvariety contained in $C_X(F)$. For any $\IT$-invariant multi-graded linear series $V_\bu$ on $X$ we have
\begin{eqnarray*} 
\delta^{g}_{Z, \IT}(X,\D;V_\bu) \ge \min\Big\{\frac{A_{X,\D}(F)}{S^{g}(V_\bu; F)},\,\, \delta^{g}_{Z,\IT}(F,\D_F;W_\bu) \Big\}, 
\end{eqnarray*}
where $W_\bu$ is the refinement of $V_\bu$ by $F$. 
\end{thm}

The theorem is an equivariant version of the Abban-Zhuang estimate \cite[Theorem 3.2]{AZ22} when $g=1$. Surprisingly, this estimate is not only useful in checking the (weighted) K-semistability but also in checking the (weighted) K-polystability, when the $\IT$-action on $F$ is trivial and the second term on the right-hand side of the estimate is strictly greater than $1$, see the proof of Theorem \ref{Theorem. stability of 2.28}. 
%We believe this method based on weighted Abban-Zhuang estimate will have a wide application on checking weighted K-stability.

The moduli spaces of Fano threefolds in the families №2.28 and №3.14 are non-trivial. 
Recall that every smooth Fano threefold in the family №2.28 is given by the blowup of $\IP^3$ along a smooth plane cubic curve $C\seq H \cong \IP^2$. If further blowing up a point outside $H$, we get a smooth Fano manifold in the family №3.14. Parallel with the proof of Theorem \ref{Theorem: Intro №2.28 and No 3.14 soliton}, we will show that for any Fano threefold $X$ in the families №2.28 and №3.14, the weighted K-stability of $(X,\xi_0)$ is equivalent to the GIT-stability of the plane cubic curves. 

\begin{thm}[Theorem \ref{Theorem: equivalence of weighted K-stability of №2.28 and №3.14 to the GIT-stability of plane cubic}]
\label{Theorem: Intro GIT and weighted K, 2.28 and 3.14}
Let $X$ be the blowup of $\IP^3$ along a plane cubic curve $C\seq H\cong \IP^2$ or further blowing up a point outside $H$. Let $\xi_0$ be the soliton candidate of $X$. Then $(X, \xi_0)$ is weighted K-semistable (weighted K-polystable) if and only if $C\seq H$ is GIT-semistable (GIT-stable or polystable). 
\end{thm}

Hence the GIT-moduli space (stack) of plane cubic curves should be viewed as a weighted K-moduli space (stack) of Fano threefolds in the families №2.28 or №3.14. 
To the best of the authors' knowledge, these may be the first examples of KRS Fano manifolds with non-trivial moduli, which give an answer to \cite[Question 5.6]{Ino19}. 

On the other hand, we have the following generalization of Koiso's theorem \cite{Koi90}, which establishes the equivalence between the K-stability of a log Fano pair $(V,\D)$ and the weighted K-stability of the cone $(Y, \D_Y, \xi_0)$ (see Section \ref{Section: Weighted K-stability of cones}) over $(V,\D_V)$. 

\begin{thm}[Theorem \ref{Theorem: soliton of cone}]
\label{Theorem: Intro stability of cones}
Let $(V, \D_V)$ be a $(n-1)$-dimensional log Fano pair such that $L=-\frac{1}{r}(K_V+\D_V)$ is an ample Cartier divisor for some $0<r\le 1$. 
Let $Y=\overline{\CC}(V,L)$ be the projective cone over $V$ with polarization $L$, and $\D_Y$ be the closure of $\D_V\times \IC^*$ in $Y$. Then $(Y, \D_Y, \xi_0)$ is weighted K-semistable (weighted K-polystable) if and only if $(V,\D_V)$ is K-semistable (K-stable or K-polystable).
\end{thm} 

By this theorem, the K-moduli space (stack) of log Fano pairs $(V,\D_V)$ should be viewed as the weighted K-moduli space (stack) of log Fano triples $(Y, \D_Y, \xi_0)$. This also gives us a series of examples of KRS Fano varieties with non-trivial moduli.

\begin{rmk} \rm 
There are strictly semistable objects in the GIT-moduli space of plane cubic curves, and the K-moduli space of log Fano pairs $(V,\D_V)$. Theorem \ref{Theorem: Intro №2.28 and No 3.14 soliton} and \ref{Theorem: Intro stability of cones} give us examples of strictly weighted K-semistable Fano varieties. 
\end{rmk}

We briefly sketch the proof of Theorem \ref{Theorem: Intro №2.28 and No 3.14 soliton}. Recall that a Fano manifold $X$ in family №2.28 is the blowup of $\IP^3$ along a plane cubic $C\seq H$, where $H$ is the plane containing $C$. We still denote  by $H$ the strict transform of the plane. The Fano manifold $X$ is destabilized by $H$, hence is K-unstable and does not admit the KE metric. We take refinement of $R_\bu=R(-K_X)$ by $H$ in the first step, and get a $\IN^2$-graded linear series $W_\bu$. From the definition of the soliton candidate, we know that $A_X(H)/S^g(R_\bu; H)=1$. Explicit computation (see Section \ref{Subsection: computing S^g of №2.28}) shows that $\delta^g_{p,\IT}(H; W_\bu)>1$ for any point $p\in H$. By Theorem \ref{Theorem: Intro weighted Abban-Zhuang}, we deduce that $(X, \xi_0)$ is weighted K-semistable. To prove the weighted K-polystability, we need the existence of the delta-minimizer established in \cite[Lemma 4.14]{BLXZ23}. Assume that $(X,\xi_0)$ is strictly K-semistable. By the above lemma, there exists a $\IT$-invariant quasi-monomial valuation $v$ which is not of the form $\wt_\xi$ for $\xi\in N(\IT)_\IR$ such that $A_X(v)/S^g(R_\bu;v)=1$. We can find a valuation $v_0$ on $H$ whose extension on $X$ is just $v$, and it's not difficult to show that $A_X(v)=A_H(v_0)$ and $S^g(R_\bu;v)=S^g(W_\bu; v_0)$. Hence 
$$1=\frac{A_X(v)}{S^g(R_\bu;v)}=\frac{A_{H}(v_0)}{S^g(W_\bu;v_0)} \ge \delta^g_{p,\IT}(H;W_\bu)>1, $$
which is a contradiction. 

\begin{rmk}\rm
There is a well-known example found in \cite{Fut83} similar to the one we introduced above, that is, the blowup of $\IP^3$ along a plane conic and a point outside the plane. This is the first example of a Fano manifold with reductive automorphism group and does not admit a K\"ahler-Einstein metric. It admits a K\"ahler-Ricci soliton by \cite{CS18} based on the structure of $\IT$-varieties of complexity one. Using the same argument as above, we will get another proof of the existence of K\"ahler-Ricci soliton. 
\end{rmk}

The paper is organized as follows. In Section \ref{Section: Preliminaries}, we recall the basic notions in K-stability, especially the $\IT$-action on a graded linear series and the weight decomposition. The concepts of weighted K-stability can be found in Section \ref{Section: Weighted K-stability}. Various formulations of the weighted expected vanishing order of a multi-graded linear series will be stated in Section \ref{Section: Weighted Abban-Zhuang estmate}. We establish the weighted Abban-Zhuang estimate (Theorem \ref{Theorem: Intro weighted Abban-Zhuang}) at the end of the section, which is the main technical results of this paper. Furthermore, we generalize the $G$-equivariant K-stability of \cite{Zhu21} to the weighted setting in Section \ref{Section: G-Equivariant weighted K-stability}. Finally, we prove the main theorems in the last three sections. In Section \ref{Section: Applications} and \ref{Section: GIT and weighted K, 2.28 and 3.14}, we prove Theorem \ref{Theorem: Intro GIT and weighted K, 2.28 and 3.14}, the equivalence between the weighted K-stability of Fano threefolds in the families №2.28, №3.14 and the GIT-stability of plane cubic curves. In Section \ref{Section: Weighted K-stability of cones}, we prove Theorem \ref{Theorem: Intro stability of cones}, the equivalence between the K-stability of a log Fano pair $(V,\D)$ and the weighted K-stability of the log Fano cone $(Y, \D_Y, \xi_0)$ over it. 

\noindent {\bf Acknowledgments}. We would like to thank our advisor, Gang Tian, for his constant support and guidance. We thank Chenyang Xu and Yuchen Liu for their suggestions on the weighted K-stability and K-moduli. We thank Akito Futaki for telling us the example of $\IP^3$ with a conic blowup. We also thank Lu Qi, Fei Si, and Shengxuan Zhou for helpful discussions. The second author was partially supported by the NKRD Program of China (\#2023YFA1010600), (\#2020YFA0713200) and LNMS.

\section{Preliminaries}
\label{Section: Preliminaries}

\subsection{Notations and Conventions}
We work over the field of complex numbers $\IC$. A {\it variety} is a separated integral scheme of finite type over $\IC$. A {\it pair} $(X, \Delta)$ consists of a normal variety $X$ and an effective $\IQ$-divisor $\Delta$ on $X$ such that $K_X+\Delta$ is $\IQ$-Cartier. A pair $(X, \Delta)$ is called {\it log Fano} if it is klt and $-K_X-\Delta$ is ample.  

Let $(X,\D)$ be a $n$-dimensional log Fano pair. Fix an integer $l_0 > 0$ such that $-l_0(K_X+\D)$ is Cartier. We denote by 
$R\coloneqq \oplus_{m\in l_0\IN} R_m=:R(X,\D)$ the anti-canonical ring of $(X,\D)$ where $R_m \coloneqq H^0(X, -m(K_X+\D))$. The number $m$ is always assumed to be a multiple of $l_0$.

\subsection{Valuations}
\label{Subsection. Valuations}

Let $K$ be a field. An $\IR$-{\it valuation} $v$ on $K$ is a function $v: K^\times \to \IR$ such that $v(fg)=v(f) + v(g), v(f+g)\ge \min\{v(f), v(g)\}$ for all $f, g \in K^\times$. For convenience, we set $v(0)=+\infty$. The {\it trivial valuation} $v_{\triv}$ is defined as $v_{\triv}(f)=0$ for all $f\in K^\times$. 

Let $X$ be a normal variety. A {\it valuation} $v$ on $X$ is an $\IR$-valuation on the rational function field $K(X)$ with a center on $X$ and $v|_{\Ik^\times}=0$. Recall that the {\it center} of $v$, denoted by $c_X(v)$, is a scheme-theoretic point $\zeta$ on $X$ such that $v\ge 0$ on $\CO_{X,\zeta}$ and $v>0$ on the maximal ideal $\fm_{\zeta}\seq \CO_{X,\zeta}$. We denote by $C_X(v)=\overline{c_X(v)}\subseteq X$ the corresponding closed irreducible subscheme on $X$. If $X$ is proper, then every valuation $v$ has a unique center on $X$. 

Let $\pi:Y\to X$ be a birational morphism between normal varieties. Any prime divisor $E\seq Y$ determines a {\it divisorial valuation} $\ord_E$ on $X$ (we also say that $E$ is a divisor {\it over} $X$). The log discrepancy $A_{X,\D}(\ord_E)=A_{X,\D}(E) = \ord_E(K_Y - \pi^*(K_X+\D)) + 1$ is defined for such a valuation. 
One may define the {\it log discrepancy} $A_{X,\D}(v)$ for any  valuation $v$ on $X$, see \cite[Section 5.1]{JM12} and \cite[Definition 1.34]{Xu23}. 
We denote by $\Val_X^\circ$ the set of non-trivial valuations on $X$. If $(X, \D)$ admits a torus $\IT=\IG_m^r$ action, we denote by $\Val_X^{\IT}$ the set of $\IT$-invariant valuations on $X$, and $\Val_X^{\IT, \circ}\coloneqq \Val_X^\circ\cap \Val_X^\IT$. 

Let $(X,\D)$ be a pair. A divisor $E$ over $(X,\D)$ is said to be of {\it plt type} if there exists a projective birational morphism $\pi: Y\to X$ such that $Y$ is normal, $E\seq Y$ is a $\IQ$-Cartier $\pi$-antiample ($-E$ is $\pi$-ample) prime divisor, and $(Y,\pi_*^{-1}\D \vee E)$ is plt in a neighbourhood of $E$ (where $(\sum_i a_i D_i) \vee (\sum_i b_i D_i) = \sum_i \max\{a_i,b_i\}D_i$ for prime divisors $\{D_i\}$). If $(X,\D)$ is klt, then $\pi$ is called a {\it plt blowup} of $X$.

\subsection{Filtrations} \label{Subsection: filtration}
A {\it filtration} $\CF$ on $R$ is collection of subspaces $\CF^\lam R_m \subseteq R_m$ for each $\lam \in \IR$ and $m\ge 0$ such that
\begin{itemize}
\item {\it Decreasing.} $\CF^\lam R_m \supseteq \CF^{\lam'}R_m $ for  $\lam \le \lam'$; 
\item {\it Left-continuous.} $\CF^\lam R_m=\CF^{\lam-\epsilon}R_m$ for $0<\epsilon \ll 1$; 
\item {\it Bounded.} $\CF^\lam R_m = R_m$ for $\lam \ll 0$ and $\CF^\lam R_m = 0$ for $\lam \gg 0$; 
\item {\it Multiplicative.} $\CF^\lam R_m \cdot \CF^{\lam'}R_{m'} \subseteq \CF^{\lam+\lam'}R_{m+m'}$. 
\end{itemize}
Since $R$ is finitely generated and $\CF$ is bounded and multiplicable, there is a constant $C>0$ such that $\CF^{-mC}R_m=R_m$ for all $m$. A filtration $\CF$ is called {\it linearly bounded} if there is a constant $C>0$ such that $\CF^{mC}R_m=0$ for all $m$. 
For any valuation $v$ on $X$, there is a filtration $\CF_v$ on $R$ defined by
$$\CF_v^\lam R_m \coloneqq \{s\in R_m\mid v(s)\ge \lam\}. $$
If $A_{X,\D}(v)<+\infty$, then by the same argument of \cite[Lemma 3.1]{BJ20}, the induced filtration $\CF_v$ is linearly bounded. In particular, the trivial valuation induces the trivial filtration
$\CF_{\triv}^0 R_m = R_m,\,\,\CF_{\triv}^{>0}R_m= 0. $

Let $L$ be a line bundle on $X$ and $V\seq H^0(X, L)$ be a subspace of finite dimension. A filtration $\CF$ on $V$ is simply defined to be a decreasing sequence $\{\CF^\lam V\}$ of subspaces which is left-continuous in $\lam$ (in particular, it is given by finitely many subspaces of $V$). For any $s \in V\setminus \{0\}$, we define $\ord_\CF(s) \coloneqq \max\{\lam\mid s\in \CF^\lam R_m\}$. We also set $\ord_\CF(0)=+\infty$ for convenience. 
\begin{defi} \rm
Let $N=\dim V$. A basis $\{s_1, \cdots, s_{N}\} $ of $V$ is said to be {\it compatible} with $\CF$ if $\CF^\lam V$ is generated by those $s_j$ with $\ord_\CF(s_j) \ge \lam$ for any $\lam \in \IR$. 
\end{defi}
The lifting of any basis of $\gr_\CF V=\oplus_\lam \CF^\lam V / \CF^{>\lam}V$ to $V$ is a basis compatible with $\CF$. 
Moreover, for any two filtrations, there exists a basis compatible with both the two filtrations.
\begin{lem}\cite[Lemma 3.1]{AZ22} 
For any two filtrations $\CF, \CG$ on a finite-dimensional vector space $V$, there exists a basis of $V$ that is compatible with both $\CF$ and $\CG$. 
\end{lem}

\begin{defi} \rm
\label{Definition: basis type divisor and compatible} 
Let $V\seq H^0(X, L)$ be a subspace for some line bundle $L$ on $X$. A {\it basis type divisor} $D$ of $V$ is of the form $D=\sum_i\{s_i=0\}$ where $\{s_i\}$ is a basis of $V$. And $D$ is said to be {\it compatible} with a filtration $\CF$ on $V$ if the associated basis is. 
\end{defi}

By the above lemma, for any two filtrations $\CF, \CG$ on $V$, there always exists a basis type divisor $D$ of $V$ which is compatible with both $\CF$ and $\CG$.

%\subsection{Test configurations} Let $(X,\D)$ be a projective klt pair, and $L$ be a $\IQ$-Cartier ample divisor. A {\it test configuration (TC)} of $(X,\D,L)$ is a collection $(\CX, \D_\CX,\CL,\eta)$ consisting of                                      \begin{itemize}                                     \item A variety $\CX$ with a $\IG_m$-action generated by a holomorphic vector field $\eta\in \Hom(\IG_m, \Aut(\CX))$;                                     \item A $\IG_m$-equivariant morphism $\pi: \CX\to \IA^1$, where the $\IG_m$-action on $\IA^1$ is standard;                                      \item A $\IG_m$-equivariant $\pi$-ample $\IQ$-Cartier divisor $\CL$ on $\CX$;                                      \item A $\IG_m$-equivariant trivialization over the punctured affine line $i_\eta:(\CX,\CL)|_{\pi^{-1}(\IG_m)}\cong (X,L)\times \IG_m$, which is compatible with $\pi$ and $\pr_1$. And $\D_\CX$ is the closure of $i_\eta^{-1}(\D\times\IG_m)$ in $\CX$.  \end{itemize}  If $\CX$ is a normal variety, then $(\CX,\D_\CX,\CL,\eta)$ is called a {\it normal test configuration}. In the log Fano case, we always choose $L=-K_X-\D$,  and a normal TC $(\CX,\D_\CX,\CL,\eta)$ is called {\it (weakly) special} if $(\CX,\CX_0+\D_\CX)$ is (lc) plt, and $\CL=-K_{\CX/\IA^1}-\D_\CX + c\CX_0$ for some $c\in\IQ$. Note by adjunction that $(\CX, \D_\CX,\CL)$ being special is equivalent that the central fiber $(\CX_0, -K_{\CX_0}-\D_{\CX,0})$ is log Fano. 

\subsection{Torus actions and moment polytopes} \label{Subsection:torus action}
Let $(X, \D)$ be a log Fano pair of dimension $n$ with a $\IT$-action. We say that the $\IT$-action is of {\it complexity} $k$ if the $\IT$-orbit of a general point in $X$ is of codimension $k$. Assume that $(X,\D)$ admits a $\IT=\IG^r_m$-action ($r=n-k$). Then the $\IT$-action has a canonical lifting to $R=R(X,\D)$, and we have a weight decomposition $R_m=\oplus_{\alpha\in M}R_{m,\alpha}$ where $M=M(\IT)=\Hom(\IT,\IG_m)$ is the weight lattice of the $\IT$-action. Let 
\begin{eqnarray*}
\PGamma(R_\bu)
&\coloneqq& \{(m,\alpha)\in \IN\times M\mid R_{m,\alpha}\ne0\}, \\
\BP_m(R_\bu)
&\coloneqq& (\{m\}\times M)\cap\PGamma(R_\bu), \\
\BP(R_\bu)
&\coloneqq&(\{1\}\times M_\IR) \cap \Cone(\PGamma(R_\bu)) = \overline{\bigcup_{\it{m}} \frac{1}{\it{m}} \BP_{\it{m}}(R_\bu)}. 
\end{eqnarray*}
The closed convex body $\BP=\BP(R_\bu)\seq M_\IR$ is called the {\it moment polytope} of the $\IT$-action on $R=R_\bu$. We also denote by $N=N(\IT) = \Hom(\IG_m, \IT)$ the lattice of one-parameter subgroup (1-PS), or the coweight lattice of the $\IT$-action. Any $\xi \in N_\IR$ can be viewed as a holomorphic vector field on $(X, \D)$.

\subsection{The $\xi$-twist of a valuation} 
We shortly recall the $\xi$-twist of a valuation introduced by \cite[Section 2.4]{Li19}. This is a fundamental notion in the reduced uniform K-stability which is equivalent to the K-polystability by \cite{LXZ22}. 
Let $(X,\D)$ be a log Fano pair with a $\IT$-action. Then there exists a $\IT$-equivariant dominant birational map $\pi: X \dashrightarrow Z$, where $Z$ is the Chow quotient of $X$ and $\IT$ acts on $Z$ trivially. The function field $K(X)$ of $X$ is the fractional field of $K(Z)[M] = \oplus_{\alpha \in M} K(Z) \cdot 1^\alpha$. For any valuation $\mu$ on $Z$ and $\xi \in N_\IR$ we define the $\IT$-invariant valuation $v_{\mu, \xi}$ on $X$ such that 
$$v_{\mu, \xi}(f)
= \min_\alpha\{\mu(f_\alpha)+\la \alpha, \xi \ra\},$$
for any $f=\sum_\alpha f_\alpha \cdot 1^\alpha \in K(Z)[M]$. By \cite[Lemma 4.2]{BHJ17} we know that any $\IT$-invariant valuation on $X$ is obtained in this way, and we get a non-canonical isomorphism $\Val^\IT_X \cong \Val_Z \times N_\IR$. For any $v = v_{\mu, \xi_0}\in \Val^\IT_X$ and $\xi \in N_\IR$, we define the $\xi$-twist of $v$ by $v_\xi\coloneqq v_{\mu,\xi_0+\xi}$. One can check that the definition is not dependent on the choice of the birational map $X\dashrightarrow Z$. If $\mu$ is the trivial valuation, then we denote by $\wt_\xi = v_{\mu, \xi}$.

\subsection{Multi-graded linear series}
In this subsection, we recall the notion of multi-graded linear series
introduced by \cite[Section 4.3]{LM09}, see also \cite[Definition 2.9]{AZ22}. 
Fix integers $0\le l < n$. Let $X_l$ be a normal variety of dimension $n-l$, and $L,L_1, \cdots, L_l$ be a sequence of line bundles on $X_l$. A {\it $\IN\times \IN^l$-graded linear series} $W_\bu$ associated to those $L_i$ is a collection of finite dimensional subspaces
\begin{eqnarray*}
W_{m,\beta} 
\seq H^0(X_l, mL+\beta_1L_1+\cdots+\beta_lL_l), 
\end{eqnarray*} 
for $(m,\beta)=(m,\beta_1,\cdots,\beta_l)\in \IN\times \IN^l$ such that $W_0 = \IC$ and $W_{m,\beta}\cdot W_{m',\beta'} \seq W_{m+m',\beta+\beta'}$. For convenience, we denote by $W_{m,\bu}$ the collection of spaces $\{W_{m,\beta}\}_\beta$. We also use $W_{(1,\beta)}$ to denote the graded linear series $\{W_{m(1,\beta)}=W_{m,m\beta}\}_m$. 

We define the semigroup $\QGamma(W_\bu)=\{(\it{m},\beta)\in \IN\times\IN^l\colon W_{\it{m},\beta} \ne 0\}$ and the $m$-slicing $\BQ_m(W_\bu)=(\{m\}\times \IN^l)\cap\QGamma(W_\bu)$. The {\it base} of $W_\bu$ is defined by $\BQ(W_\bu)=(\{1\}\times \IR^l) \cap \Cone(\QGamma(W_\bu))$. We say that $W_\bu$ has {\it bounded support} if $\BQ(W_\bu)$ is bounded. 

The series $W_\bu$ {\it contains an ample series} if there exists $(1,\beta)\in\IQ\times\IQ^l$ in the interior of $\BQ=\BQ(W_\bu)$ such that $L+\beta_1L_1+\cdots+\beta_lL_l = A+E$, where $A$ is an ample $\IQ$-divisor and $E$ is an effective $\IQ$-divisor, and $H^0(X_l,mA)\seq W_{m(1,\beta)}$ for sufficiently divisible $m$. 

%We use the following definition that appeared in \cite{Xu23}. We say that $W_\bu$ is { \it asymptotically complete} with respect to line bundles $L,L_1,\cdots,L_l$ if \begin{eqnarray*} \mathop{\lim}_{m\to \infty} \frac{h^0(W_{m(1,\beta_1,\cdots,\beta_l)})}{h^0(X_l,mL+m_1\beta_1L_1+\cdots+m_l\beta_lL_l)} = 1, \end{eqnarray*}for all $\beta=(\beta_1,\cdots,\beta_l) \in \interior(\BQ)_\IQ$, that is, $\IQ$-point in the interior of $\BQ$.  

In the following, we will always assume that the multi-graded linear series $W_\bu$ has bounded support and contains an ample series. 

The volume of $W_\bu$ is defined as
$\vol(W_\bu)=\limsup_{m\to \infty}\frac{\dim(W_{m,\bu})}{m^n/n!}.$ We also define the fiberwise volume $\vol(W_{(1,\beta)})=\limsup_{m\to \infty}\frac{\dim(W_{m(1,\beta)})}{m^{n-l}/(n-l)!}$ for any $\beta\in\interior(\BQ)_\IQ$. A filtration $\CF$ on $W_\bu$ is a descending sequence of subspaces $\{\CF^\lam W_{m,\beta}\}_{\lam\in\IR}$, which is bounded, left-continuous and multiplicative (see Section \ref{Subsection: filtration}). We say that a filtration $\CF$ is linearly bounded if there exists $C>0$ such that $\CF^{-Cm}W_{m,\beta}=W_{m,\beta}$ and $\CF^{Cm}W_{m,\beta}=0$ for all $m,\beta$. 

\noindent {\bf Notation remark.} 
We usually use the Abban-Zhuang estimate successively to deduce a lower bound of the weighted stability threshold of a $n$-dimensional Fano variety $X$. In the first step, we will take refinement of the anti-canonical ring $R$ of $X$, and the dimension of the Okounkov body of the multi-graded linear series in each step does not change. So we should consider the $\IN\times\IN^l$-graded linear series on a $(n-l)$-dimensional variety $X_l$ as the $l$-th step in the program. 

Let $\pi:Y_l\to X_l$ be a proper birational morphism between normal varieties, then it naturally to view $W_\bu$ on $X$ as a multi-graded linear series on $Y$ via the isomorphism 
$$H^0(X_l, mL+\beta_1L_1+\cdots+\beta_lL_l)
\cong H^0(Y_l, \pi^*(mL+\beta_1L_1+\cdots+\beta_lL_l)). 
$$
Hence the multi-graded linear series is independent of the choice of birational models.

\subsection{Okounkov bodies}
%Equip $\IZ^{n-l}$ with the lexicographic ordering. A valuation $\nv:K(X_l)^\times\to \IZ^{n-l}$ is {\it faithful} if it is surjective. By \cite[Section 2.2]{KK12} (see also \cite[Section 1.1.2]{Xu23}), one may pick an admissible flag over $X_l$ to define a faithful valuation on $K(X_l)$. 
%Here we give an another construction of a faithful valuation using plt-type divisors over $X_l$. With this method we could give an effective construction of the faithful valuation which is compatible with a torus action introduced by \cite{HL20}. 

Recall that an {\it admissible flag} $Y_\bu$ over $X_l$ is defined as a flag of subvarieties
$$Y_\bu: Y=Y_0 \supseteq Y_1 \supseteq \cdots \supseteq Y_{n-l} $$
on some projective birational model $Y\to X$, where $Y_i$ is an irreducible subvariety in $Y$ of codimension $i$ (in particular, $Y_{n-l}$ is a closed point since $\dim X_l=n-l$) that is smooth at $Y_{n-l}$. For any $\IQ$-divisor $L$ on $X$ that is Cartier at $Y_{n-l}$, one can define a faithful valuation $\nv = \nv_{Y_\bu}$ as \cite[Introduction]{LM09} (see also \cite[Definition 2.8]{AZ22}),  
$$\nv=\nv_{Y_\bu} : H^0(X,L)\setminus \{0\} \to \IN^l, s \mapsto (\nv_1(s),\cdots, \nv_{n-l}(s)). $$ 
For any $\IN\times\IN^l$-graded linear series $W_\bu$ on $X_l$ associated to line bundles $L,L_1,\cdots,L_l$, we define the semigroup
\begin{eqnarray*}
\Gamma(W_\bu)=\{(m,\beta,\nv(s))\in \IN\times\IN^l\times \IN^{n-l} \mid 0\ne s\in W_{m,\beta}\}, 
\end{eqnarray*} 
and the $m$-slicing $\BO_m(W_\bu) = \Gamma(W_\bu) \cap (\{m\} \times \IN^n)=\nv(W_{m,\bu}\setminus \{0\})$. We denote the natural projection by $q:\BO_m\to \BQ_m$. We know that $\# \Gamma_m(W_\bu) = \dim W_{m,\bu}$ (see, for example, \cite[Lemma 1.6]{Xu23}). The {\it Okounkov body} of $W_\bu$ is defined by $\BO(W_\bu)= (\{1\} \times \IR^n)\cap \Cone(\Gamma(W_\bu))$. In other word,
\begin{eqnarray*}
\BO(W_\bu)=\overline{\bigcup_{m\in \IN}\frac{1}{m}\BO_m(W_\bu)}. 
\end{eqnarray*}

Let $\CF$ be a linearly bounded filtration on $W_\bu$. Then for each $t \in \IR$ we have a $\IN\times\IN^l$-graded linear series $\CF^{(t)} W_\bu$ defined by $(\CF^{(t)} W)_{m,\beta}=\CF^{mt}W_{m,\beta}$. Note that $\CF^{(t)} W_\bu$ is linearly bounded and contains an ample series since $W_\bu$ does. We denote the Okounkov body of $\CF^{(t)} W_\bu$ by $\BO^{(t)}$, and let $\BO=\BO(W_\bu)$. Then $\BO^{(t)}\seq \BO$ is a descending collection of convex bodies. The linear boundedness of $\CF$ guarantees that $\BO^{(-C)}=\BO$ and $\BO^{(C)}=0$. 

The {\it concave transform} of $\CF$ is the function on $\BO$ defined by $G^\CF(y)=\sup\{t\mid y\in\BO^{(t)}\}$. Note that $G^\CF$ is concave and continuous. If $(\beta,\gamma) \in \BO_\IQ$, then 
$$G^\CF(\beta,\gamma)=\sup\left\{\frac{\ord_\CF(s)}{m} \mid  \frac{\nv(s)}{m} = \gamma, \ \exists \, m \in \IN,\,\exists\, s\in W_{m,m\beta} \right\}. $$

\subsection{Multi-graded linear series with a torus action} \label{Subsection: multi-graded linear series with a torus action}
Let $(X_l, \D_l)$ be a $(n-l)$-dimensional pair with a $\IT=\IG^r_m$-action, and $L,L_1,\cdots, L_l$ are $\IT$-linearized line bundles on $X_l$. Let
$V_\bu$ be a $\IT$-invarant $\IN\times\IN^l$-graded linear series with respect $L,L_1,\cdots,L_l$, that is, $V_{m,\beta}\seq H^0(X_l, mL+\beta_1L_1+\cdots+\beta_lL_l)$ is $\IT$-invariant. 
Hence each $V_{m,\beta}$ admits a weight decomposition $V_{m,\beta} = \oplus_{\alpha\in M} V_{m,\alpha,\beta}$, where the direct sum is taken over $\BP_{m,\beta}=\{\alpha\in M\mid V_{m,\alpha,\beta}\ne 0\}$. We also denote by 
\begin{eqnarray*}
\PGamma(V_\bu)
&\coloneqq& \{(m,\alpha)\in \IN\times M\mid V_{m,\alpha,\beta}\ne0, \exists \beta \}, \\
\BP_m(V_\bu)
&\coloneqq& (\{m\}\times M)\cap\PGamma(V_\bu), \\
\BP(V_\bu)
&\coloneqq&(\{1\}\times M_\IR) \cap \Cone(\PGamma(V_\bu)) = \overline{\bigcup_{\it{m}} \frac{1}{\it{m}} \BP_{\it{m}}(V_\bu)}. 
\end{eqnarray*}
We will write $V_{m,\bu,\beta}$ instead of $V_{m,\beta}$.  Hence $V_m=\oplus_{\beta\in \BQ_m} V_{m,\bu,\beta} =\oplus_{\beta\in \BQ_m,\alpha\in \BP_{m,\beta}} V_{m,\alpha,\beta}$. 
For $\beta \in \interior(\BQ)_\IQ$, the moment polytope of the graded linear series $V_{(1,\beta)}$ is $\BP_{(1,\beta)}=\overline{\cup_m\frac{1}{m}\BP_{m,m\beta}}\seq M_\IR$, which is a slicing of $\BP=\BP(V_\bu)$. 
We also set $\BQ_{m,\alpha}\coloneqq\{\beta\in\IN^l\mid V_{m,\alpha,\beta}\ne 0\}$ and $V_{m,\alpha,\bu}=\oplus_{\beta\in\BQ_{m,\alpha}} V_{m,\alpha,\beta}$. So 
$V_m=\oplus_{\alpha\in \BP_m}
V_{m,\alpha,\bu}$ is the weight decomposition of $V_\bu$. 

\begin{rmk}\rm
The $\IT$-action on $V_\bu$ is always assumed to be faithful. But the $\IT$-action on $X$ may have a non-trivial stabilizer at general points. See Section \ref{Section: Applications} for examples. 
We also remark that $\alpha$ and $\beta$ may not be independent, see Example \ref{Example: formula when Q=P}. 
\end{rmk}

\subsection{Refinements}
\label{Subsection: Refinements and toric filtrations}
Let $(X_l,\D_l)$ be a klt pair with $\IT$-action and $V_\bu$ be a $\IT$-invariant $\IN\times\IN^l$-graded linear series on $X_l$. For any $\IT$-invariant $\IZ$-filtration $\CF$ on $V_\bu$, we define the {\it refinement} of $V_\bu$ along $\CF$ by $W_\bu\coloneqq\gr_\CF V_\bu$, which is a $\IT$-invariant $\IN\times\IN^{l+1}$-graded linear series. More precisely, $W_{m,\bu,\beta, j}=\CF^jV_{m,\bu,\beta}/\CF^{j+1}V_{m,\bu,\beta}$. 
If moreover $\CF=\CF_E$ for some $\IT$-invariant plt-type divisor $E$ over $(X_l,\D_l)$, then $W_\bu$ is a multi-graded linear series on $E$ by \cite[Example 2.6]{AZ22}. 

Let $E$ be a toric divisor over $X_l$, that is, $\ord_E = \wt_\xi$ for some $\xi\in N$. 
Let $b = \inf\{t\in \IR\mid \la \alpha,\xi \ra + mt \ge 0, \forall \alpha \in \BP(V_\bu) \}$, which is a rational number since $\BP(V_\bu)$ is rational. 
In this case, we have $\ord_E(s_\alpha) = \wt_\xi(s_\alpha) = \la\alpha, \xi\ra + mb \ge 0$ for any $s_\alpha \in V_{m,\alpha,\bu}$. In particular, we have $\CF\coloneqq\CF_E = \CF_{\xi}(b)$. Let $\Phi:\PGamma(V_\bu) \to \IQ$ be the linear function $\Phi(m,\alpha) = \la \alpha,\xi \ra + mb$. Then for any $\beta\in{\rm int}(\BQ)_\IQ$ and sufficiently divisible $m$, we have 
%The filtration $\CF$ is toric, if there is a linear function $\Phi:\PGamma(V_\bu) \to \IZ$ (that is, $\Phi(m,\alpha)= A(\alpha) + mb$, where $A:\IZ^r\to\IZ$ is a linear functional and $b\in \IZ$ is a constant) such that 
$$\CF^j V_{m,\bu,\beta} = 
\bigoplus_{\alpha\in \BP_{m,\beta},\Phi(m,\alpha)\ge j} V_{m,\alpha,\beta}. 
$$
%If $\CF=\CF_{\xi}(b)$ for some $\xi\in N_\IQ$ and $b \in \IQ$, then $\Phi(m,\alpha) = \la \alpha,\xi \ra + mb$ for sufficiently divisible $m$. 
Hence 
$$W_{m,\bu,\beta,j} = \CF^j V_{m,\bu,\beta}/\CF^{j+1}V_{m,\bu,\beta} \cong \bigoplus_{\alpha\in \BP_{m,\beta},\Phi(m,\alpha) =j} V_{m,\alpha,\beta}. $$ 
For any $\alpha\in \BP_{m,\beta}$ with $\Phi(m,\alpha)= j$, we should denote by
$$W_{m,\alpha,\beta,j}\coloneqq V_{m,\alpha,\beta}/\CF^{j+1}V_{m,\alpha,\beta} \cong V_{m,\alpha,\beta}.  $$ 
Hence $W_{m,\bu,\beta, j}
=\oplus_{\alpha\in \BP_{m,\beta,j}(W_\bu)} W_{m,\alpha,\beta,j}$ where $\BP_{m,\beta,j}(W_\bu)=\{\alpha\in \BP_{m,\beta}\mid \Phi(\alpha)= j\}\seq \BP_{m,\beta}(V_\bu)$. On the other hand, let 
$\BQ_{m,\alpha}(W_\bu)
= \{(\beta,\Phi(m,\alpha))\in \IN^{l+1}\mid \beta \in \BQ_{m,\alpha}\}
= \BQ_{m,\alpha}(V_\bu)\times \{\Phi(m,\alpha)\}\seq \IN^{l+1}$. In conclusion 
\begin{eqnarray*} 
\BP(W_\bu)
&=&\overline{\bigcup_{m,\beta,j}\frac{1}{m} \BP_{m,\beta,j}(W_\bu)}
=\overline{\bigcup_{m,\beta}\frac{1}{m} \BP_{m,\beta}(V_\bu)}
=\BP(V_\bu), \\
\BQ(W_\bu)
&=&\overline{\bigcup_{m,\alpha}\frac{1}{m} \BQ_{m,\alpha}(W_\bu)}\to \BQ(V_\bu), \quad
(\beta, j) \mapsto \beta. 
\end{eqnarray*}
The moment polytope does not change and the base polytope admits a one-dimensional fibration over the origional base polytope, whose fiber over $\beta \in \BQ(V_\bu)$ is $\Phi(1, \overline{\cup_m\frac{1}{m}\BP_{m,m\beta}})$.

\begin{rmk} \rm
The anti-canonical ring $R$ of a $n$-dimensional log Fano pair $(X,\D)$ with a $\IT=\IG^r_m$-action admits a weight decomposition $R_m=\oplus_{\alpha}R_{m,\alpha}$, which should be viewed as a $\IN\times\IN^r$-graded linear series on a $(n-r)$-dimensional $\IT$-invariant subvariety $Z$ over $X$ on which the $\IT$-action is trivial. Such a subvariety $Z$ will be obtained by taking refinements by toric divisors successively. 
\end{rmk}

\section{Weighted K-stability}
\label{Section: Weighted K-stability}
%In this section, we define the {\it weighted stability threshold} $\delta^g(X,\D)$, and we show that it is equal to the {\it equivariant weighted stability threshold} $\delta^g_\IT(X,\D)$ defined in [BLXZ23], that is, $\delta^g_\IT=\delta^g$. 

In this section, we define the invariants about the {\it $g$-weighted K-stability}. The invariants are almost the same as those of the ordinary K-stability. The only difference is that we replace the measures in the integrals defining the invariants by the {\it $g$-weighted measures}. 

Let $(X, \D)$ be a log Fano pair with a $\IT$-action. Then there is a canonical lifting of the $\IT$-action on the anti-canonical ring $R_\bu=R(X,\D)$. We denote the $1$-PS lattice by $N=N(\IT)$ and the weight lattice by $M=M(\IT)$. Let $\BP\seq M_\IR$ be the moment polytope of the  $\IT$-action on $R_\bu$, and $\BO\seq\IR^n$ be the Okounkov body of $-(K_X+\D)$ which is compatible with the $\IT$-action (see for example \cite[Definition 2.21]{HL20}). Then there is a linear projection $p:\IR^n\to M_\IR$ mapping $\BO$ onto $\BP$. The $\IT$-action on $R_\bu$ gives the weight decomposition
$R_m=\oplus_{\alpha \in M\cap m\BP} R_{m,\alpha}$.

\subsection{Soliton candidates}
\label{Subsection: soliton candidate}
In many examples, the main reason that leads to a Fano manifold being K-unstable is that there exists a product test configuration destabilizing it. It's natural to ask how can we remove the defect of such a test configuration. 
The {\it modified Futaki invariant} $\Fut_g$ introduced by \cite{TZ02} is the natural candidate to give an affirmative answer to this question. We state an algebraic definition of $\Fut_g$ here. 

Let $g: \BP\to \IR$ be a continuous function. 
We define
$$\Fut_g(\xi)=-\frac{1}{\Bv^g}
\int_\BP \la\alpha, \xi\ra g(\alpha) \DH_\BP(\dif \alpha), \,\,\xi\in N_\IR$$
where $\Bv^g=\int_\BP g(\alpha) \DH_\BP(\dif \alpha)$ is the {\it $g$-weighted volume} of $-(K_X+\D)$ and $\DH_\BP$ is the Duistermaat-Heckman measure of $-(K_X+\D)$ on $\BP$ (see Section \ref{Subsection: Sg via Okounkov body}).   
Fix a $\xi_0\in N_\IR$, and let $g(\alpha)=e^{-\la \alpha, \xi_0 \ra}$ for any $\alpha\in\BP$, then $\Fut_g$ reveals the modified Futaki invariant. 
In this case, $\Fut_g$ is the first order derivative of the so-called $\BH$-functional at $\xi=\xi_0$ (see for example \cite[Section 2.5]{HL20}), 
$$\BH(\xi)= \log\Big(
\int_\BP e^{-\la\alpha,\xi\ra} \DH_\BP(\dif \alpha)\Big), \,\,\xi\in N_\IR. $$
Note that $\BH: N_\IR \to \IR$ is a strictly convex function on $N_\IR$ and tends to $+\infty$ when $|\xi|\to +\infty$. Hence there exists a unique $\xi_0\in \IN_\IR$ minimizing $\BH$. For this $\xi_0$, we have $\Fut_g(\xi)=0$ for all $\xi\in N_\IR$. The vector field $\xi_0$ is called the {\it soliton candidate} of $(X,\D)$ with respect to the $\IT$-action, and we simply say that $(X,\D,\xi_0)$ is a {\it log Fano triple}. 
If $(X,\D)$ is K-semistable, then $\xi_0=0$ and $\Fut_g$ is the ordinary Futaki invariant. 

\begin{rmk} \rm 
\label{Remark. choice of g}
Through out this paper, the function $g:\BP\to \IR_{>0}$ is assumed to be continuous and $\Fut_g|_N = 0$. The function $g$ is called the {\it weight function}, for which we can define the {\it $g$-weighted K-stability} (see Definition \ref{Definition: T-equivariantly weighted K-semistable}). Whenever considering a log Fano triple $(X,\D,\xi_0)$, we always assume that $g(\alpha) = e^{-\la\alpha,\xi_0\ra}$. 
\end{rmk}

\subsection{Weighted expected vanishing order via DH measures}
Let $\CF$ be a $\IT$-invariant linearly bounded filtration on $R_\bu$. We simply denote by 
\begin{eqnarray*} 
N_m &=&
\dim R_m, \\
N_{m,\alpha} &=& 
\dim R_{m,\alpha}, \\
N^\lam_m &=& 
\dim \CF^\lam R_m - \dim \CF^{>\lam} R_m \\
N^\lam_{m,\alpha} &=&
\dim \CF^\lam R_{m,\alpha} - \dim \CF^{>\lam} R_{m,\alpha}. 
\end{eqnarray*} 
Then we define the asymptotic invariant (see \cite[Section 6]{RTZ21})
\begin{eqnarray*} 
S^g_m(R_\bu; \CF) 
&\coloneqq& 
\frac{1}{N^g_m} 
\sum_{\lam\in\IR} 
\frac{\lam}{m}
\sum_{\alpha\in \BP_m} 
g(\frac{\alpha}{m}) 
N^\lam_{m,\alpha},   
\end{eqnarray*}
where $N^g_m 
\coloneqq \sum_{\alpha\in \BP_m} 
g(\frac{\alpha}{m}) 
N_{m,\alpha} =: m^n \Bv^g_m$. One may choose a basis $\{s_j\}_{j=1}^{N_m}$ of $R_m$ which is compatible with the weight decomposition and the filtration $\CF$ (exists since $\CF$ is $\IT$-invariant) with $\lam_j=\ord_\CF(s_j)$ such that 
$\lam_1\ge \lam_2\ge \cdots \ge \lam_{N_m}.$
The collection $\{\lam_j\}$ is called the {\it ($m$-th) successive minima} of $\CF$. 
We also set $\alpha_j=\alpha$ if $s_j\in R_{m,\alpha}$. Note that $\{\alpha_j\}\seq M, \{\lam_j\}\seq \IR$ are both independent of the compatible basis. Then 
\begin{eqnarray*} 
S^g_m(R_\bu; \CF) 
&=& 
\frac{1}{N^g_m} 
\sum_{j} 
\frac{\lam_j}{m}\cdot
g(\frac{\alpha_j}{m}). 
\end{eqnarray*}
Consider the following discrete measure on $\IR$ determined by $\CF$
\begin{eqnarray*} 
\DH^g_{\CF,m}
&=& 
\frac{1}{m^n} 
\sum_{\lam\in\IR} 
\sum_{\alpha\in \BP_m} 
g(\frac{\alpha}{m}) 
N^\lam_{m,\alpha} \cdot
\delta_{\frac{\lam}{m}}.  
\end{eqnarray*} 
Hence we have $\Bv^g_m=\int_\IR \DH^g_{\CF,m}$ and 
\begin{eqnarray*} 
S^g_m(R_\bu; \CF) 
&=& 
\frac{1}{\Bv^g_m} \int_\IR t \cdot \DH^g_{\CF, m}(\dif t). 
\end{eqnarray*}
We define the $g$-volume of graded linear series $\{\CF^{(t)}R_{\bu}\}$ as 
\begin{eqnarray*}
    \vol^g(\CF^{(t)}R_{\bu}) = \mathop{\limsup}_{m\rightarrow \infty} \sum_{\alpha \in \BP_m} g(\frac{\alpha}{m})\frac{\dim \CF^{mt}R_{m,\alpha}}{m^n/n!}.
\end{eqnarray*}
In fact the limit exists and we define the Duistermaat-Heckman ($\DH$) measure 
\begin{eqnarray*}
    \DH_{\CF}^g(\dif t)= -\frac{1}{n!} \dif\,\vol^g(\CF^{(t)}R_{\bu})
\end{eqnarray*}
as a derivative distribution. The following lemma is mentioned in \cite[Section 2.6]{HL20} (for the $g=1$ case, see \cite[Theorem 1.11]{BC11} and \cite[Theorem 2.8]{BJ20}).
\begin{lem}\label{lemma:DHgWC}
    Suppose $\CF$ is a linearly bounded filtration of $R_{\bu}$, then the sequences of measures $\DH_{\CF,m}^g$ converges weakly to $\DH_{\CF}^g$ as $m\rightarrow \infty$.
\end{lem}
\begin{proof}
    Let $f_{m,g}(t) \coloneqq \frac{1}{m^n}\sum_{\alpha \in \BP_m}g(\frac{\alpha}{m})\dim \CF^{mt}R_{m,\alpha}$ and $f_g(t) \coloneqq (n!)^{-1}\vol^g(\CF^{(t)}R_{\bu})$. According to the definition of $g$-volume of a graded linear series, it follows that $f_{m,g}(t)$ converges to $f_g(t)$ pointwise. Note that the moment polytope $\BP=\BP(R_{\bu})$ is compact and the weight function $g\colon \BP \rightarrow \IR_{>0}$ is continuous, then there exists a constant $C = C(g, \BP)>0$ such that $g(x)\leq C$ for any $x \in \BP$. Therefore, $|f_{m,g}(t)|\leq \frac{C}{m^n}N_m$ are uniformly bounded. By the dominated convergence theorem, we obtain that $f_{m,g}(t)$ converges to $f_g(t)$ in $L^{1}_{\loc}(\IR)$. 

    Observe that $\DH_{\CF,m}^g(\dif t) = -\frac{\dif}{\dif t} f_{m,g}(t)$ and $\DH_{\CF}^g(\dif t) = -\frac{\dif}{\dif t}f_g(t)$ as distributions. Then $L^{1}_{\loc}(\IR)$ convergence of $f_{m,g}(t)$ implies $\DH_{\CF,m}^g$ converges to $\DH_{\CF}^g$ as distributions. By \cite[Theorem 2.1.9]{Hor03}, we conclude $\DH_{\CF,m}^g$ converges to $\DH_{\CF}^g$ weakly as measures.
\end{proof}
Let $\Bv^g \coloneqq \int_\IR \DH^g_\CF$. We define the {\it $g$-weighted expected vanishing order} of $\CF$ by 
\begin{eqnarray*} 
S^g(R_\bu; \CF) 
&\coloneqq& 
\frac{1}{\Bv^g} \int_\IR t \cdot \DH^g_\CF(\dif t). 
\end{eqnarray*} 
We also denote it by $S^g(R;E)$ or $S^g(R;v)$ if $\CF$ is induced by a prime divisor $E$ or a valuation $v$ respectively. 
It's clearly that $\Bv^g_m$ and $S^g_m(R_\bu; \CF)$ converge to $\Bv^g, S^g(R_\bu; \CF)$ respectively. 

\begin{rmk}\rm 
In the literature (for example, \cite[Section 2.1.4]{BLXZ23}), the measure $\DH_\CF$ is always defined to be a probability measure. However, by our definition, $\DH^g_\CF$ is not a probability measure. It has total mass $\Bv^g = \int_{\IR}\DH^g_\CF$. If $g=1$, then $\Bv^g=\Bv=\vol(\BO(L))$ is the volume of the Okounkov body of $L=-(K_X+\D)$.  
\end{rmk}

\subsection{Weighted expected vanishing order via Okounkov bodies}
\label{Subsection: Sg via Okounkov body}
We define the following discrete measures on the Okounkov body $\BO$ and the moment polytope $\BP$ respectively 
%$$N^g_m= \sum_{\alpha\in M \cap m\BP} g(\frac{\alpha}{m})h^0(R_{m,\alpha}), $$
\begin{eqnarray*} 
\LE_m
&=&\frac{1}{m^n}\sum_{\gamma\in \BO_m}
\delta_\frac{\gamma}{m}, \\
\DH_{\BP,m}
&=&\frac{1}{m^n}\sum_{\alpha\in \BP_m}  
N_{m,\alpha} \cdot
\delta_\frac{\alpha}{m}.
\end{eqnarray*}
We see that $\DH_{\BP,m}=p_*\LE_m$. By \cite[Lemma 1.4]{Xu23}, the discrete measure $\LE_m$ converges weakly to the Lebesgue measure $\LE$ on $\BO$. Pushing forward by $p:\BO\to \BP$, we see that $\DH^g_{\BP,m} \to \DH^g_\BP$ also converges weakly. We define the following {\it $g$-weighted measures} on $\BO$ and $\BP$  
\begin{eqnarray}
\LE^g=g\circ p\cdot \LE,\quad 
\DH^g_\BP=g\cdot \DH_\BP, 
\end{eqnarray}
respectively. Then $\DH^g_\BP=p_*\LE^g$. We also define $\LE^g_m=g\circ p\cdot \LE_m$ and $\DH^g_{\BP,m}=g\cdot \DH_{\BP,m}$. 
Then the same argument as above shows that they converge weakly to $\LE^g$ and $\DH^g_\BP$ respectively. Let $\BO' \seq \BO$ be a convex body. The {\it $g$-weighted volume} of $\BO'$ is $\vol^g(\BO')\coloneqq\int_{\BO'}\LE^g$. Clearly we have $\int_\BO \LE^g_m = N^g_m/m^n=\Bv^g_m$, and the $g$-weighted volume of $\BO$ is just 
\begin{eqnarray*}
\Bv^g=\vol^g(\BO)
=\int_{\BO}\LE^g
=\int_{\BP}\DH^g_\BP. 
\end{eqnarray*}

For any $\IT$-invarant linearly bounded filtration $\CF$ on $R=R_\bu$, let $G^\CF$ be the concave transform of the filtration $\CF$. Note that $\BO(\CF^{(t)}R_{\bu}) = \{x\in \BO(R_{\bu})\colon G^{\CF}(x)\geq t\}$. Then we have 
\begin{eqnarray*} 
\DH^g_\CF(\dif t) 
= -\frac{1}{n!}\cdot\dif\, \vol^g(\CF^{(t)}R_{\bu})=
-\dif \, \vol^g\{G^\CF \ge t\} 
\end{eqnarray*}
as distributions.

Integrating by part, we get 
$
S^g(R_\bu; \CF) = 
\frac{1}{\Bv^g}
\int_\IR 
\vol^g\{G^\CF\ge t\}\dif t 
$. Hence 
\begin{eqnarray*} 
S^g(R_\bu; \CF) 
&=& 
\frac{1}{\Bv^g} \int_\BO G^\CF \cdot \LE^g. 
\end{eqnarray*} 
If $\CF=\CF_E$ is induced by a toric divisor $E$ (see Section \ref{Subsection: Refinements and toric filtrations}), then $G^\CF$ is invariant on each fiber of $p$, hence descends to a function on $\BP$ (still denoted by $G^\CF$). Hence 
\begin{eqnarray*} 
S^g(R_\bu; \CF) 
&=& 
\frac{1}{\Bv^g} 
\int_\BP G^\CF 
\cdot \DH^g_{\BP}. 
\end{eqnarray*} 
With this formualtion of $S^g$, we may also define the asymptotic invariants $\hat{S}^g_m$ analogous to $S^g_m$ by 
\begin{eqnarray*} 
\hat{S}^g_m(R_\bu; \CF) 
\,\,\coloneqq\,\,
\frac{1}{\Bv^g_m} \int_\BO G^\CF \cdot\LE^g_m
\,\,=\,\, 
\frac{1}{N^g_m} \sum_{j} 
G^\CF(\frac{\nv(s_j)}{m}) \cdot
g(\frac{\alpha_j}{m}). 
\end{eqnarray*}
By definition of $G^\CF$, we have $\frac{\ord_\CF(s_j)}{m} \le G^\CF(\frac{\nv(s_j)}{m})$. Hence $S^g_m(R_\bu; \CF)\le \hat{S}^g_m(R_\bu; \CF)$. 

We have the following analog of \cite[Corollary 2.10]{BJ20}, see also \cite{BLXZ23}. For the convenience of the reader, we state a proof in the appendix. 

\begin{lem} \label{Lemma: uniform bound of S^g_m/S^g}
For any $\varepsilon>0$ there exists $m_0\in\IN$ such that $S^g_m(R_\bu; \CF) \le (1+\varepsilon)S^g(R_\bu; \CF)$ for any $m\ge m_0$ and any linearly bounded filtration $\CF$ on $R_\bu$. 
\end{lem}

\subsection{Weighted stability thresholds}
\label{Subsection. Weighted stability thresholds}
Following \cite{RTZ21, BLXZ23}, we define 
\begin{eqnarray*} 
\delta^g_{\IT,m}(X,\D)
\,\,\coloneqq\,\, 
\inf_v \frac{A_{X,\D}(v)}{S^g_m(R_\bu;v)}, \quad
\delta^g_\IT(X,\D)
\,\,\coloneqq\,\, 
\inf_v \frac{A_{X,\D}(v)}{S^g(R_\bu;v)},
\end{eqnarray*}
where the infimum runs over all the valuations $v\in \Val^{\IT,\circ}_X$. Using the same argument of \cite[Theorem 4.4]{BJ20}, we have 
\begin{lem}\cite[Proposition 6.14]{RTZ21}, \cite[Lemma 4.5]{BLXZ23} \label{Lemma: convergence of delta_m}
\begin{eqnarray*}
    \mathop{\lim}_{m\to \infty}\delta^g_{\IT,m}(X,\D)
= \delta^g_\IT(X,\D).
\end{eqnarray*}
\end{lem}
\begin{proof}
    It follows from Lemma \ref{lemma:DHgWC} that $\lim_{m\rightarrow \infty}S_m^g(R_{\bu};\CF) = S^g(R_{\bu};\CF)$ holds for any linearly bounded filtration $\CF$ on $R_{\bu}$.  Recall that a valuation $v\in \Val_X$ with $A_{X,\Delta}(v)<+\infty$ induces a linearly bounded filtration $\CF_{v}$.  We get $\limsup_{m\rightarrow \infty}\delta_{\IT,m}^g(X,\Delta) \leq \delta_{\IT}^g(X,\Delta)$ from the pointwise convergence of $S$-invariant. On the other hand, according to Lemma \ref{Lemma: uniform bound of S^g_m/S^g}, given $\varepsilon>0$ there exists $m_0\in \IN$ such that $S_m^g(R_{\bu};\CF)\leq (1+\varepsilon)S^g(R_{\bu};\CF)$ for any $m\geq m_0$ and any $v\in \Val_X$ with $A_{X,\D}(v)<\infty$. Thus 
    \begin{eqnarray*}
        \mathop{\liminf}_{m\rightarrow \infty}\delta_{\IT,m}^g(X,\D) = \mathop{\liminf}_{m\rightarrow \infty}\mathop{\inf}_{v\in \Val_X} \frac{A_{X,\D}(v)}{S_m^g(R_{\bu};v)}\geq (1+\varepsilon)^{-1}\mathop{\inf}_{v\in \Val_X}\frac{A_{X,\D}(v)}{S^g(R_{\bu};v)}. 
    \end{eqnarray*}
    Letting $\varepsilon \rightarrow 0$, we get $ \mathop{\liminf}_{m\rightarrow \infty}\delta_{\IT,m}^g(X,\D) \geq \delta_{\IT}^g(X,\D)$. This completes the proof.
\end{proof}
%\begin{eqnarray*} \delta^g_m(X,\D) = \inf_D \lct(X,\D;D), \\ \delta^g(X,\D) = \inf_D \lct(X,\D;D), \end{eqnarray*} where the first infimum runs over all the $\IT$-invariant $g$-weighted $m$-basis type $\IR$-divisors of $R_\bu$, the second one needs moreover to run over all the $m\in \IN$. 

\begin{defi} \rm
\label{Definition: T-equivariantly weighted K-semistable}
A log Fano triple $(X,\D, \xi_0)$ with a $\IT$-action is called {\it $\IT$-equivariantly weighted K-semistable} if $A_{X,\D}(v) - S^g(R_\bu;v) \ge 0$ for any $v\in\Val_{X}^\IT$. It is {\it $\IT$-equivariantly weighted K-polystable} if it is $\IT$-equivariantly weighted K-semistable, and $A_{X,\D}(v) - S^g(R_\bu;v) = 0$ implies that $v=\wt_\xi$ for some $\xi\in N_\IR$. 
\end{defi}

This definition is equivalent to \cite[Definition 4.10]{BLXZ23} by \cite[Proposition 4.14]{BLXZ23}. 

\begin{rmk}\rm
Whenever we talk about the weighted K-semistability of a log Fano triple $(X,\D,\xi_0)$, a torus action containing the soliton candidate $\xi_0$ is fixed. We shall omit the words ``$\IT$-equivariantly'' if the $\IT$-action is clear. 
\end{rmk}

\section{Weighted Abban-Zhuang estmate}
\label{Section: Weighted Abban-Zhuang estmate}
In this section, we deal with the invariants of weighted K-stability associated with multi-graded linear series. They are the straightforward generalizations of the invariants in the previous section. We mainly focus on the various formulations which are useful in explicit computations. At the end of the section, we establish the weighted Abban-Zhuang estimate, which gives a lower bound of the $g$-weighted stability threshold.

\subsection{Weighted expected vanishing order}
We work with the same assumption as Section \ref{Subsection: multi-graded linear series with a torus action}. Let $(X_l, \D_l)$ be a $(n-l)$-dimensional pair with a $\IT=\IG^r_m$-action, and $V_\bu$ be a $\IT$-invarant $\IN\times\IN^l$-graded linear series on $X_l$. 
We simply denote by $N_*$ the dimension of $V_*$ where $*=m, (m,\alpha,\bu), (m,\bu,\beta)$ or $(m,\alpha,\beta)$. Let $\CF$ be a $\IT$-invariant linearly bounded filtration on $V_\bu$. We also denote by $N^\lam_*$ the dimension of $\gr^\lam_\CF V_*$ for $*$ as above. Then the definition of $S^g_m(V_\bu;\CF)$ follows
\begin{eqnarray*}
S^g_m(V_\bu; \CF) 
&\coloneqq& 
\frac{1}{\Bv^g_m} \int_\IR t \cdot \DH^g_{\CF, m}(\dif t), 
\end{eqnarray*}
where $\Bv^g_m = \int_\IR \DH^g_{\CF, m}(\dif t)$ and 
\begin{eqnarray*} 
\DH^g_{\CF,m}
&\coloneqq& 
\frac{1}{m^n} \sum_{\alpha\in \BP_m} 
\sum_{\lam\in\IR} 
g(\frac{\alpha}{m}) 
N^\lam_{m,\alpha,\bu}
\cdot \delta_{\frac{\lam}{m}}. 
\end{eqnarray*} 
Using the same argument as Lemma \ref{lemma:DHgWC}, we get $\DH^g_{\CF,m}$ converges weakly to the DH measure $\DH^g_\CF$ on $\IR$. We define $\Bv^g=\lim_{m\to \infty} \Bv^g_m$ and 
\begin{eqnarray*}
S^g(V_\bu; \CF) 
&\coloneqq& 
\frac{1}{\Bv^g} \int_\IR t \cdot \DH^g_{\CF}(\dif t), 
\end{eqnarray*}

Let $\BO\seq \IR^n, \BQ\seq \IR^l, \BP\seq \IR^r$ be the Okounkov body, base polytope, and moment polytope of $V_\bu$ respectively. Denote by $q:\BO\to \BQ$ the natural linear projection. We may assume that the Okounkov body $\BO$ is compatible with the $\IT$-action, then there exists a natural linear projection $p:\BO\to \BP$. We have the discrete measures on $\BO$ and $\BP$ respectively 
\begin{eqnarray*} 
\LE^g_m
&=&\frac{1}{m^n}\sum_{\gamma\in \BO_m}
g(p(\frac{\gamma}{m}))\cdot
\delta_\frac{\gamma}{m},  \\
\DH^g_{\BP,m}
&=&\frac{1}{m^n}\sum_{\alpha\in \BP_m}
g(\frac{\alpha}{m}) \dim(V_{m,\alpha,\bu})
\cdot\delta_\frac{\alpha}{m}.
\end{eqnarray*} 
By \cite[Lemma 1.4]{Xu23}, they converge weakly to $\LE^g$ and $\DH^g_\BP$ respectively. We also have $\Bv^g=\int_\BO\LE^g=\int_\BP\DH^g_\BP=\vol^g(\BO)$, and $\Bv^g_m = \int_\BO \LE^g_m=\int_\BP \DH^g_{\BP,m}=N^g_m/m^n$. 

Let $G^\CF$ be the concave transform of $\CF$, which is a concave function on $\BO$. We have 
\begin{eqnarray} \label{Formula: S^g 2}
S^g(V_\bu; \CF) 
\,\,=\,\,
\frac{1}{\Bv^g} \int_\IR t \cdot \DH^g_\CF
\,\,=\,\,
\frac{1}{\Bv^g} \int_\IR \vol^g\{G\ge t\}\dif t
\,\,=\,\,
\frac{1}{\Bv^g} \int_\BO G^\CF \cdot \LE^g. 
\end{eqnarray}

We may compute $S^g$ using the multi-graded structure. Let $\beta \in \BQ_\IQ$, and denote by $\BO^{(t)}=\{G^\CF\ge t\}$ the Okounkov body of $\CF^{(t)}V_\bu$. Then $q^{-1}(\beta)\cap \BO^{(t)}$ is the Okounkov body of the $\IN$-graded linear series $\CF^{(t)}V_{(1,\bu,\beta)}$, and we have 
\begin{eqnarray*} 
\vol^g(\BO^{(t)}) 
&=& 
\int_\BQ \vol^g(q^{-1}(\beta)\cap \BO^{(t)}) \LE_\BQ(\dif \beta), 
\end{eqnarray*}
where $\LE_\BQ(\dif \beta)$ is the Lebesgue measure on $\BQ$. Hence
\begin{eqnarray} \label{Formula: S^g 3}
S^g(V_\bu; \CF) 
&=&
\frac{1}{\Bv^g} 
\int_\IR
\int_\BQ \vol^g(q^{-1}(\beta)\cap \BO^{(t)}) \LE_\BQ(\dif \beta)
\dif t, 
\end{eqnarray}

\begin{ex}\rm 
\label{Example: formula when Q=P} 
Assume that the projection $p:\BO\to \BP$ factor through $q:\BO\to \BQ$, that is, there exists a linear projection $\bar{p}:\BQ\to \BP$ such that $p=\bar{p}\circ q$. The weighted volume is clear in this case 
\begin{eqnarray*} 
\vol^g(q^{-1}(\beta)\cap \BO^{(t)}) = g(\bar{p}(\beta)) \cdot
\vol(q^{-1}(\beta)\cap \BO^{(t)}). 
\end{eqnarray*}
Changing the order of the integration, we get
\begin{eqnarray}
\label{Formula: S^g 4}
S^g(V_\bu; \CF) 
&=&
\frac{1}{\Bv^g} 
\int_\BP g(\alpha)
\Big(
\int_\IR
\int_{\bar{p}^{-1}(\alpha)} 
\vol(q^{-1}(\beta)\cap \BO^{(t)}) \LE_{\bar{p}^{-1}(\alpha)}(\dif \beta)\dif t
\Big)
\LE_\BP(\dif \alpha) 
\end{eqnarray}
Recalling the definition of the $S$-invariant (formula (\ref{Formula: S^g 3}) with $g=1$), we further have
\begin{eqnarray} 
\label{Formula: S^g 5}
S^g(V_\bu; \CF) 
&=&
\frac{1}{\Bv^g} 
\int_\BP g(\alpha)
\vol(p^{-1}(\alpha))S(V_{(1,\alpha,\bu)};\CF)
\LE_\BP(\dif \alpha). \\
&=&
\frac{1}{\Bv^g} 
\int_\BP S(V_{(1,\alpha,\bu)};\CF)
\DH^g_\BP(\dif \alpha). 
\end{eqnarray}
%We make a remark here that $\vol(V_{(1,\alpha, \bu)})=(n-r)! \vol(p^{-1}(\alpha))$. 
This is the key formula in our proof of Theorem \ref{Theorem: Intro №2.28 and No 3.14 soliton} and \ref{Theorem: Intro stability of cones}. 
\end{ex}

\subsection{Weighted basis type divisors}
We use the same notions as above. 
\begin{defi} \rm
A {\it $\IT$-invariant $g$-weighted $m$-basis type $\IR$-divisor} $D$ of $V_\bu$ is of the form 
\begin{eqnarray} 
\label{Formula: weighted basis type divisor}
D=
\frac{1}{N^g_m} \sum_{\alpha\in\BP_m} g(\frac{\alpha}{m})
 \frac{D_{m,\alpha,\bu}}{m}, 
\end{eqnarray}
where $D_{m,\alpha,\bu}= \sum_{\beta\in\BQ_{m,\alpha}} D_{m,\alpha,\beta}$ and $D_{m,\alpha,\beta}$ is a basis type divisor of $V_{m,\alpha,\beta}$. 
We say that $D$ is {\it compatible} with a filtration $\CF$ on $V_\bu$ if every $D_{m, \alpha,\beta}$ is compatible with $\CF$ on $V_{m,\alpha,\beta}$. 
\end{defi}

From the definition we see that, if $D$ is compatible with $\CF_v$ for some valuation $v$, then we have 
$$v(D)=S^g_m(V_\bu; \CF_v).$$

\begin{rmk}\rm
The weighted basis type divisor is the key notion in the proof of the weighted Abban-Zhuang estimate (Theorem \ref{Theorem: weighted AZ}). 
\end{rmk}

Let $F$ be a $\IT$-invariant plt-type divisor (Section \ref{Subsection. Valuations}) over $X_l$ and $\pi: Y\to X_l$ be the associated plt-type blowup. For any $\IT$-invariant multi-graded linear series $V_\bu$, we denote by $W_\bu$ the refinement of $V_\bu$ by $\CF=\CF_F$. We consider a $\IT$-invariant $g$-weighted $m$-basis type $\IR$-divisor $D$ which is compatible with $F$ and has decomposition as (\ref{Formula: weighted basis type divisor}). 

If $F$ is a toric divisor (that is, $\ord_F=\wt_\xi$ for some $\xi\in N(\IT)$), then $W_{m,\alpha,\beta,\Phi(\alpha)} = V_{m,\alpha,\beta} \cdot s_F^{-\Phi(\alpha)} |_F$ locally (for $m$ sufficiently divisible such that $\Phi(\alpha) \in \IZ$ for all $\alpha\in \BP_m$), where $s_F\in \CO_X$ is the local defining function of $F$. Globally we have $\pi^*D_{m,\alpha,\beta} = \Phi(\alpha) N_{m,\alpha,\beta} \cdot F + \Gamma_{m,\alpha,\beta,\Phi(\alpha)}$, where $\Gamma_{m,\alpha,\beta,\Phi(\alpha)}$ does not contain $F$ as a component, and $\Gamma_{m,\alpha,\beta,\Phi(\alpha)}|_F$ is a basis type divisor of $W_{m,\alpha,\beta,\Phi(\alpha)}$. Hence we have 
\begin{eqnarray*} 
\Gamma
&=& 
\frac{1}{N^g_m}
\sum_{\alpha\in\BP_m(V_\bu)} 
g(\frac{\alpha}{m}) 
%\sum_{\beta\in\BQ_{m,\alpha}(V_\bu)}
\frac{1}{m} \cdot
\Gamma_{m,\alpha,\bu,\Phi(\alpha)}, \\
\pi^*D
&=&
\Big(
\frac{1}{N^g_m} 
\sum_{\alpha\in \BP_m(V_\bu)} 
g(\frac{\alpha}{m}) 
\frac{\Phi(\alpha)}{m} N_{m,\alpha,\bu} 
\Big) \cdot F
+ \Gamma\\
&=& S^g_m(V_\bu; F) \cdot F
+ \Gamma, 
\end{eqnarray*}
where 
$\BP_m(W_\bu)=\BP_m(V_\bu),
\BQ_{m,\alpha}(W_\bu)
=\BQ_{m,\alpha}(V_\bu) 
\times \{\Phi(\alpha)\}. $
Hence $\Gamma|_F$ is a $\IT$-invariant $g$-weighted $m$-basis type $\IR$-divisor of $W_\bu$. 

If $F$ is vertical, then $W_{m,\alpha,\beta,j}=\CF^jV_{m,\alpha,\beta}\cdot s_F^{-j}|_F$ locally. Hence we have $D_{m,\alpha,\beta}=\sum_jD_{m,\alpha,\beta,j}$, where $\ord_F(D_{m,\alpha,\beta,j})=j$. Then 
$$D_{m,\alpha,\beta,j}
=jN^j_{m,\alpha,\beta}\cdot F + 
\Gamma_{m,\alpha,\beta,j}, $$ 
where $\Gamma_{m,\alpha,\beta,j}$ does not contain $F$ as a component. Summing up we also have
\begin{eqnarray*} 
\pi^*D
&=& S^g_m(V_\bu; F) \cdot F + \Gamma, \\
\Gamma
&=& 
\frac{1}{N^g_m}
\sum_{\alpha\in\BP_m(W_\bu)} 
g(\frac{\alpha}{m}) 
\frac{1}{m} \cdot 
\Gamma_{m,\alpha,\bu}, \\
\Gamma_{m,\alpha,\bu}
&=&
\sum_{(\beta,j)\in\BQ_{m,\alpha}(W_\bu)}
\Gamma_{m,\alpha,\beta,j}
\end{eqnarray*}
Hence $\Gamma|_F$ is a $\IT$-invariant $g$-weighted $m$-basis type $\IR$-divisor of $W_\bu$. 

\begin{rmk} \rm
We conclude that, for any $\IT$-invariant $g$-weighted $m$-basis type $\IR$-divisor $D$ of $V_\bu$ which is compatible with $F$, we have decomposition 
\begin{eqnarray}
\label{Formula: refinement of weighted basis type divisor}
\pi^*D
&=& S^g_m(V_\bu; F) \cdot F + \Gamma, 
\end{eqnarray}
where $\Gamma$ does not contain $F$ as a component and $\Gamma|_F$ is a $\IT$-invariant $g$-weighted $m$-basis type $\IR$-divisor of $W_\bu$. 
\end{rmk}

\subsection{Local weighted stability thresholds}
\begin{defi} \rm
\label{Definition: local delta}
Let $U$ be a quasi-projective variety with a $\IT$-action, $(X,\D)$ be a klt pair with a $\IT$-action, and $f:(X, \D)\to U$ be a $\IT$-equivariant projective morphism. We denote by $V_\bu$ a $\IT$-invariant multi-graded linear series on $X$. For any $\IT$-invariant subvariety $Z\seq X$. We define 
\begin{eqnarray*} 
\delta^g_{Z, \IT, m}(X,\D;V_\bu) 
= \inf_{v} \frac{A_{X,\D}(v)}{S^g_m(V_\bu; v)}, \quad
\delta^g_{Z, \IT}(X,\D;V_\bu) 
= \inf_{v} \frac{A_{X,\D}(v)}{S^g(V_\bu; v)}, 
\end{eqnarray*}
where the infimum runs over all the valuations $v\in \Val^{\IT,\circ}_X$ passing through $Z$, that is, $C_X(v)\supseteq Z$. 
\end{defi}

For a $\IT$-invariant valuation $v$ passing through $Z$, we have $S^g_m(V_\bu; v)=\sup_D v(D)$ where the supremum runs over all the $\IT$-invariant $g$-weighted $m$-basis type $\IR$-divisor $D$ of $V_\bu$. For such a divisor $D$, the local log canonical threshold  
$$\lct_Z(X,\D;D)
\coloneqq\sup\{t\in \IR\mid (X,\D+tD) \text{ is lc at the generic point of } Z\}
= \mathop{\inf}_{v:C_X(v)\supseteq Z} \frac{A_{X,\D}(v)}{v(D)}$$ 
is minimized by a $\IT$-invariant valuation. Hence 
\begin{eqnarray} \label{formula: delta_m = lct}
\delta^g_{Z, \IT, m}(X,\D;V_\bu) 
= \inf_{v} \inf_D
 \frac{A_{X,\D}(v)}{v(D)}
=\inf_D \lct_Z(X,\D;D), 
\end{eqnarray}
where the infimum runs over all the $\IT$-invariant $g$-weighted $m$-basis type $\IR$-divisor $D$ of $V_\bu$. Similarly, we have $\delta^g_{Z, \IT}(X,\D;V_\bu) =\inf_D \lct_Z(X,\D;D)$, where the infimum runs over all the $\IT$-invariant $g$-weighted $m$-basis type $\IR$-divisor $D$ of $V_\bu$ for all $m\in \IN$.

\subsection{Weighted Abban-Zhuang estmate}
Consider $f:(X,\D)\to U$ and $V_\bu$ as above. Let $F$ be a plt-type divisor over $X$, with the associated plt-type blowup $\pi: Y \to X$. 
We denote by $W_{\bu}$ the refinement of $V_\bu$ by $F$, and by $\D_Y$ be the strict transform of $\D$ on $Y$. Then we have 
\begin{eqnarray*}
K_Y+\D_Y+(1-A_{X,\D}(F))F=\pi^*(K_X+\D).
\end{eqnarray*}
We denote by $\D_F$ the difference of $\D_Y$ on $F$, then
\begin{eqnarray*}
(K_Y+\D_Y+F)|_F=K_F+\D_F. 
\end{eqnarray*}

\begin{thm} \label{Theorem: weighted AZ}
Let $Z \subseteq C_X(F)$ be a $\IT$-invariant subvariety. Then we have 
\begin{eqnarray} \label{ineq.1}
\delta^{g}_{Z, \IT}(X,\D;V_\bu) \ge \min\Big\{\frac{A_{X,\D}(F)}{S^{g}(V_\bu; F)},\,\, \mathop{\inf}_{\pi(Z')=Z} \delta^{g}_{Z',\IT}(F,\D_F;W_\bu) \Big\}. 
\end{eqnarray}
\end{thm}

\begin{rmk}\rm
If $F$ is a divisor on $X$, we may set $Y=X, \D=\D_Y+\ord_F(\D)F$, and assume that $(Y,\D_Y+F)$ is lc in a neighbourhood of $F$. The theorem still holds in this case, with $\delta^g_Z(F,\D_F;W_\bu)$ replaced by $\delta^g_Z(\hat{F},\D_{\hat{F}};W_\bu)$, where $\hat{F}$ is the normalization of $F$. 
\end{rmk}

The theorem is a strengthened version of \cite[Theorem 3.2]{AZ22}, which is very useful in explicit computations. We will state a proof in the next section under a more general setting.

\section{$G$-Equivariant weighted K-stability}
\label{Section: G-Equivariant weighted K-stability}

Let $G$ be an algebraic group. In this section, we prove the equivalence of the $G$-equivariant $g$-weighted K-semistability and the $\IT$-equivariant $g$-weighted K-semistability for a log Fano triple $(X,\D,\xi_0)$ with a $G$-action, where $\IT\seq G$ is a subtorus containing $\xi_0$. This is an application of \cite{Zhu21} to the {\it $g$-weighted boundaries} on $X$.

\subsection{Equivariant Abban-Zhuang estimate of boundaries}
We first recall the notion of boundaries introduced by \cite{Zhu21}. 

Let $(X, \D)$ be a pair. A {\it boundary} $V$ on $X$ is a formal sum $V=\sum_i a_i V_i$, where $a_i \in \IR_{>0}$ and $V_i\seq H^0(X,L_i)$ are subspaces of finite dimensions for some line bundle $L_i$ on $X$.  A {\it basis type ($\IR$-)divisor} of $V$ is of the form $D=\sum_i a_i D_i$, where $D_i=\frac{1}{\dim V_i}\sum_j\{s_j=0\}$, and $\{s_j\} \seq V_i$ is a basis. Then we define $c_1(V)$ by $\sum_i a_i L_i \sim_\IR D$. For any filtration $\CF$ on $V$, we have $S(V;\CF)= \sum_i a_i \sum_\lam \lam \cdot \frac{\dim \gr^\lam_\CF V_i}{\dim V_i}$. 

Let $F$ be a plt-type divisor over $X$ with the associated plt-type blowup $\pi: Y\to X$. The {\it refinement} of $V$ by $F$ is defined as 
$$W=\sum_i a_i \sum_\lam \frac{\dim \gr^\lam_\CF V_i}{\dim V_i}\cdot V_i(-\lam F)|_F, $$
where $V_i(-\lam F)|_F \cong \gr^\lam_\CF V_i$ is the image of $\CF_F^\lam V_i$ under the restriction map $H^0(Y, \pi^*L_i-\lam F)\to H^0(F, (\pi^*L_i-\lam F)|_F)$. For any basis type divisor $D$ of $V$, we have
\begin{eqnarray} \label{Formula: refinement of basis type divisor of a boundary}
\pi^*D=S(V;F)F + \Gamma,  
\end{eqnarray}
where $\Gamma$ does not contain $F$ as a component, and $\Gamma|_F$ is a basis type divisor of $W$. We make a remark that any basis type divisor of $W$ is obtained in this way. 

Assume moreover that $(X,\D)$ admits an algebraic group $G$-action, and the line bundles $L_i$ admit $G$-linearization such that $V_i\seq H^0(X,L_i)$ are $G$-invariant ($V=\sum_i a_i V_i$ is called a {\it $G$-invariant boundary} in this case). Let $U$ be a $G$-variety and $f: (X, \D)\to U$ be a $G$-equivariant morphism. For a $G$-invariant subvariety $Z\seq X$, we shall define $\delta_{Z,G}(V)=\delta_{Z,G}(X,\D,V) = \inf_v \frac{A_{X,\D}(v)}{S(V;v)}$ as Definition \ref{Definition: local delta}, where the infimum runs over all the $G$-invariant valuations $v$ on $X$ passing through $Z$. 

If $G=\IT$ is a torus, we have the following equivariant version of the Abban-Zhuang estimate. 
\begin{lem} \label{Theorem: weighted AZ 2}
For any $\IT$-invariant plt-type divisor $F$ over $X$ and $\IT$-invariant subvariety $Z\seq C_X(F)$, we have
$$
\delta_{Z, \IT}(X,\D;V) \ge \min\Big\{\frac{A_{X,\D}(F)}{S(V; F)},\,\, \mathop{\inf}_{\pi(Z')=Z}\delta_{Z',\IT}(F,\D_F;W) \Big\}, 
$$
where $W$ is the refinement of $V$ by $F$. 
\end{lem}
\begin{proof}
We denote by $\lam$ the right hand side of the inequality. By \cite[Proposition 3.1]{AZ22}
$$\delta_{Z,\IT}(X,\D;V)=\delta_{Z,\IT}(X,\D;\CF_F;V)=\inf_D\lct_Z(X,\D;D), $$ 
where the infimum runs over all the $\IT$-invariant basis type divisor $D$ of $V$ compatible with $F$. Hence it suffices to show that $(X,\D+\lam D)$ is lc at the generic point of $Z$ for any $D$ as above. 

Let $\pi:Y\to X$ be the associated plt-type blowup of $F$, and $\D_Y$ be the strict transform of $\D$ to $Y$. Then 
$\pi^*(K_X+\D)=K_Y+\D_Y+(1-A_{X,\D}(F))F$, and $(K_Y+\D_Y+F)|_F=K_F+\D_F$. For any $\IT$-invariant basis type divisor $D$ of $V$ compatible with $F$, we have
$\pi^*D = S(V; F)\cdot F + \Gamma$ by (\ref{Formula: refinement of basis type divisor of a boundary}), where $\Gamma$ is the strict transform of $D$ on $Y$, and $F \nsubseteq \Supp \Gamma$. So we have 
\begin{eqnarray}
\label{Formula: adjunction pair in AZ 2}
\pi^*(K_X+\D+\lam D)=K_Y+\D_Y+(1-A_{X,\D}(F)+\lam S(V;F))F +\lam\Gamma. 
\end{eqnarray}
Note that $1-A_{X,\D}(F)+\lambda S(V; F) \le1$. Hence it suffices to show that $(Y,\D_Y+F+\lam \Gamma)$ is lc at the generic point of $Z$, which is equivalent that $(F, \D_F+\lambda \Gamma|_F)$ is lc at the generic point of $Z$ by inversion of adjunction. 
Since $\Gamma|_F$ is $\IT$-invariant basis type divisor of $W$, $\lct_Z(F,\D_F;\Gamma|_F)$ is minimized by a $\IT$-invariant valuation. Hence $\lct_Z(F,\D_F;\Gamma|_F) \ge \delta_{Z,\IT}(F,\D_F;W)\ge \lam$. The proof is finished. 
\end{proof}

We are ready to give a proof of the weighted Abban-Zhuang estimate (Theorem \ref{Theorem: weighted AZ}). 

\begin{defi}\rm 
\label{Definition: g-weighted boundary}
Using notions in Theorem \ref{Theorem: weighted AZ}, we define the {\it $m$-th $g$-weighted boundary} of $V_\bu$ by 
\begin{eqnarray*}
V^g_m = \sum_{\alpha, \beta}  
g(\frac{\alpha}{m})
\frac{N_{m,\alpha,\beta}}{m\cdot N^g_m} \cdot V_{m,\alpha,\beta}, 
\end{eqnarray*}
(recall that $N^g_m=\sum_{\alpha,\beta}g(\frac{\alpha}{m})N_{m,\alpha,\beta}$). 
Then the $\IT$-invariant $g$-weighted $m$-basis type $\IR$-divisor of $V_\bu$ is just the basis type divisor of $V^g_m$, and $S^g_m(V_\bu;\CF)=S(V^g_m; \CF)$ for any $\IT$-invariant filtration $\CF$ on $V_\bu$. The refinement of $V^g_m$ by $F$ is just the $m$-th $g$-weighted boundary $W^g_m$ of $W_\bu$. Hence 
\begin{eqnarray*}
\delta^{g}_{Z',\IT,m}(F,\D_F,W_\bu)
&=& 
\delta_{Z',\IT}(F,\D_F,W^g_m). 
\end{eqnarray*}
\end{defi}

\begin{proof}[Proof of Theorem \ref{Theorem: weighted AZ}]
We denote by $\lambda$ the right hand side of the inequality (\ref{ineq.1}), and let 
\begin{eqnarray*} 
\lambda_m \coloneqq \min\Big\{\frac{A_{X,\D}(F)}{S^{g}_m(V_\bu; F)},\,\,\mathop{\inf}_{\pi(Z')=Z} \delta^{g}_{Z',\IT,m}(F,\D_F,W_\bu) \Big\}. 
\end{eqnarray*}
One can show that $\lim_{m \to \infty} \lambda_m = \lambda $ by Lemma \ref{Lemma: uniform bound of S^g_m/S^g}. We conclude by Theorem \ref{Theorem: weighted AZ 2}. 
\end{proof}

\subsection{The equivalence of equivariant weighted K-semistability for different groups}
Let $G$ be an algebraic group and $(X,\D)$ be a log Fano pair with a $G$-action. Let $\xi_0$ be the soliton candidate and  and $\IT\seq G$ be a subtorus containing $\xi_0$. We show that the equivariant weighted K-semistability of $(X,\D)$ with respect to $G$ and $\IT$ are equivalent in this subsection. 

\begin{thm}
Let $(X,\D,\xi_0)$ be a log Fano triple admitting an algebraic group $G$-action, and $\IT\seq G$ be a subtorus containing $\xi_0$. If $(X,\D,\xi_0)$ is not $\IT$-equivariantly $g$-weighted K-semistable, then $\delta^g_{\IT}(X,\D)$ is achieved by a sequence of $G$-invariant weakly special divisors over $X$. 
%then $\delta^g_{\IT}(X,\D)$ is minimized by a $G$-invariant weakly special divisor over $X$. 
%The existence of minimizer is given by ? \cite{BLXZ23}
\end{thm}
\begin{proof}
We follow the proof of \cite[Theorem 4.1]{Zhu21}. Let $V_\bu=R(X,\D)$ be the anti-canonical ring of $(X,\D)$, and $\delta_m = \delta_\IT(X,\D;V^g_m)$. Hence $c_1(V^g_m)\sim_\IR -(K_X+\D)$ is ample. By Lemma \ref{Lemma: uniform bound of S^g_m/S^g}, for any $\varepsilon>0$ there exists $m_0=m_0(\varepsilon)\in\IN$ such that $S(V^g_m;\CF)\le (1+\varepsilon) S^g(V_\bu;\CF)$ holds for all $m\ge m_0$ and all the linearly bounded filtrations $\CF$ on $V_\bu$. On the other hand, we shall increase $m_0\in\IN$ such that $\delta_m<1$ for all $m\ge m_0$ by Lemma \ref{Lemma: convergence of delta_m}. Then $-(K_X+\D+\delta_mc_1(V^g_m))\sim_{\IR} -(1-\delta_m)(K_X+D)$ is ample.  
By \cite[Theorem 4.4]{Zhu21} (with $U=$ point and $V=V^g_m$), we have $G$-invariant prime divisor $E_m$ minimizing $\delta_m = \delta_\IT(X,\D;V^g_m)=\delta^g_{\IT,m}(X,\D;V_\bu)$. Hence
$$(1+\varepsilon) \delta_m = 
(1+\varepsilon)
\frac{A_{X,\D}(E_m)}{S(V^g_m;E_m)} 
\ge \frac{A_{X,\D}(E_m)}{S^g(V_\bu;E_m)} 
\ge \delta^g_{\IT}(X,\D;V_\bu). $$
Pick series $\varepsilon_i \to 0$, we have constants $m_i=m_i(\varepsilon_i)\in\IN$ as above. By Lemma \ref{Lemma: convergence of delta_m}, we conclude that $$\frac{A_{X,\D}(E_{m_i})}{S^g(V_\bu;E_{m_i})} \to \delta^g_{\IT}(X,\D;V_\bu),\, (i\to \infty).  $$ 

On the other hand, for $m\ge m_0$, the divisor $E_m$ is an lc place of $\delta_m D + (1-\delta_m) H$ where $D$ is a basis type $\IR$-divisor of $V^g_m$ compatible with $E_m$, and $H$ is a general effective divisor such that $H\sim_\IR-(K_X+\D)$. Hence $E_m$ is weakly special. 
\end{proof}

\begin{rmk} \rm
Using the argument of \cite{LXZ22}, one may further show that $\delta^g_\IT(X,\D)<1$ is minimized by a $G$-invariant weakly special divisor over $X$. 
\end{rmk}

\begin{cor}
%If $\beta^g(E)\ge 0$ for all $G$-invariant weakly special divisor $E$ over $X$, 
If $\delta^g_G(X,\D)\ge 1$, then $(X,\D,\xi_0)$ is $\IT$-equivariantly $g$-weighted K-semistable,
\end{cor}
\begin{proof}
Otherwise, by the above theorem there exists a $G$-invariant prime divisor $E$ over $(X,\D)$ such that $\frac{A_{X,\D}(E)}{S^g(V_\bu;E)} < 1$, which contradicts $\delta^g_G(X,\D)\ge 1$. 
\end{proof}

\section{K\"ahler-Ricci solitons on Fano threefolds of №2.28 and №3.14}
\label{Section: Applications}
Let $(X, \D)$ be a log Fano pair of dimension $n$ with a $\IT$-action. We say that the $\IT$-action is of {\it complexity} $k$ if the $\IT$-orbit of a general point in $X$ is of codimension $k$. 
By \cite{ACC+}, we have the following K-unstable Fano threefolds with torus actions of complexity two, 
\begin{itemize}
    \item №2.26: blow up $V_5\subset \IP^6$ along line.
    \item №2.28: blow up $\IP^3$ along plane cubic.
    \item №3.14: blow up of $\IP^3$ along plane cubic curve and point that are not coplanar.
    \item №3.16: blow up of $V_7$ along proper transform via blow up $V_7 \rightarrow \IP^3$ of twisted cubic passing through blown up point.
\end{itemize}
We will show that every smooth Fano threefold in the family {\rm №2.28} and {\rm №3.14} admits a K\"ahler-Ricci soliton in this section. 

\subsection{Weighted K-polystability of №2.28}
\label{Subsection: №2.28}
Every smooth Fano threefold $X$ of family №2.28 is given by the blowup of $\IP^3_{u,x,y,z}$ along a smooth plane cubic curve $C=\{u=0, y^2z=x(x-z)(x-\lambda z)\}$ for $\lambda \neq 0,1$. We will write $v\coloneqq y^2z-x(x-z)(x-\lambda z)$ for simplicity. Let $\pi:X\to \IP^3$ be the natural morphism and $E_C\seq X$ be the exceptional divisor. There is a natural $\IG_a^3$-action and $\IG_m$-action on $X$, they are given by
\begin{eqnarray*}
\mu: \IG_a^3 \times \IP^3 \to \IP^3, &\quad& 
\big( (a,b,c), [u,x,y,z] \big) \mapsto [u, x+au, y+bu, z+cu], \\
\lam: \IG_m \times \IP^3 \to \IP^3, &\quad&
\big( t, [u,x,y,z] \big) \mapsto [t^{-1}u,x,y,z]. 
\end{eqnarray*}
The plane $H_u=\{u=0\}$ lies in the fixed locus of these group actions, hence they lift to $X$ naturally. Indeed, the identity component of the automorphism group of $X$ is $\Aut^0(X) = \IG_a^3 \rtimes \IG_m$, see \cite{ACC+}. Note that the $\IG_m$-action $\lam$ is induced by the divisorial valuation $\ord_{H_u}$, and $\lam^{-1}$ is induced by $\ord_{E_0}$, where $E_0$ is the exceptional divisor of the blowup of $\IP^3$ at $[1,0,0,0]$. Hence the moment polytope of the $\IG_m$-action is $[-1, 3]$, since $A_{\IP^3}(H_u)=1$ and $A_{\IP^3}(E_0)=3$ by \cite[Section 2.4 (5)]{Wang24}. 

In order to show the weighted K-polystability of $X$, we firstly need to determine the soliton candidate $\xi_0$. To achieve this, we construct an Okounkov body $\BO$ of $X$ compatible with the $\IG_m$-action. 
Consider the affine chart $U_{z} = \{z\neq 0\} \seq \IP^3$ and just assume that $z=1$. Then $X|_{U_z} \seq \IA^3_{u,x,y} \times \IP^1_{\zeta_0,\zeta_1}$ is defined by $\{\zeta_0v-\zeta_1u=0\}$. We further restrict to the affine chart $\{\zeta_1 \neq 0\}$, then $u = w\cdot v$, where $w\coloneqq\zeta_0/\zeta_1$. In this affine chart, the exceptional divisor $E_C=\{v=0\}$ and the strict transform of hyperplane $H_u$ is given by $\{w =0\}$, still denoted by $H_u$. Since $-K_X = \pi^* (-K_{\IP^3})-E_C$, we have 
\begin{eqnarray*}
H^0(X,-K_X) 
&=& w\cdot H^0(\IP^3,\mathcal{O}_{\IP^3}(3)) \oplus H^0(\IP^3,\mathcal{O}_{\IP^3}(1)) \\
&=& \la w u^3 \ra \oplus w u^2\cdot\la x,y,z  \ra \oplus w u \cdot \la x,y,z  \ra^2 \oplus  w \cdot \la x,y,z  \ra^3 \oplus \la x,y,z  \ra \\
&=& \la w^4v^3 \ra \oplus w^3v^2\cdot\la x,y,z  \ra \oplus w^2v\cdot \la x,y,z  \ra^2 \oplus  w \cdot \la x,y,z  \ra^3 \oplus \la x,y,z  \ra. 
\end{eqnarray*}
This decomposition is compatible with the $\IG_m$-action.
We take the admissible flag $\mathbb{A}^3_{w,y,x}\supseteq \{w=0\}\supseteq \{w=y=0\}\supseteq \{w=y=x=0\}$, which gives a faithful valuation $\mathfrak{v}: w\mapsto (1,0,0), y\mapsto (0,1,0), x\mapsto (0,0,1)$. We have $\nv(v) = (0,0,1)$. Thus
\begin{eqnarray*}
    w^4\cdot v^3 &\mapsto &(4,0,3);\\
    w^3\cdot v^2\cdot \langle x,y,z \rangle &\mapsto &(3,0,2), (3,1,2), (3,0,3);\\
    w^2\cdot v\cdot \langle x,y,z \rangle^2 &\mapsto &(2,0,1),(2,1,1),(2,0,2),(2,2,1),(2,1,2),(2,0,3);\\
    w\cdot \langle x,y,z \rangle^3&\mapsto &(1,0,0),(1,1,0),(1,0,1),(1,2,0),(1,1,1),(1,0,2),\\
    &&(1,3,0),(1,2,1),(1,1,2),(1,0,3);\\
    \langle x,y,z \rangle&\mapsto &(0,0,0),(0,1,0),(0,0,1).
\end{eqnarray*} 
Let $\BO_0\seq \IR^3_{w,y,x}$ be the convex polytope generated by these points. Computing volume we see that $\BO_0$ is an Okounkov body of $X$. In order to make the Okounkov body compatible with the $\IG_m$-action, we may shift the first coordinate by one such that the body lives in $-1\le w\le3$. We denote this shifted body by $\BO$, and the natural projection by $p:\BO\to \BP=[-1,3]$. Note that $\BO$ has vertices $\{(3,0,3), (0,0,0),(0,0,3),(0,3,0),(-1,0,0),(-1,1,0),(-1,0,1)\}$. 

From the equation $\Fut_g(1)=0$, we derive $0=\int_{\BO} w\cdot e^{-\xi_0\cdot w}\dif w \dif y \dif x = \int_{-1}^0 w\cdot e^{-\xi_0\cdot w}\cdot \frac{1}{2} (3+2w)^2\dif w + \int_0^3 w\cdot e^{-\xi_0\cdot w}\cdot \frac{1}{2} (3-w)^2 \dif w$. We can numerically solve $\xi_0 \approx  0.9377815610300645$.
 
\begin{thm}
\label{Theorem. stability of 2.28}
Every smooth Fano threefold $X=X_{2.28}$ admits a K\"ahler-Ricci soliton. 
\end{thm}
\begin{proof}
By the work of \cite{HL23} and \cite{BLXZ23}, we know that $X$ admits a K\"ahler-Ricci soliton if and only if $(X,\xi_0)$ is weighted K-polystable. We first show that it is weighted K-semistable. 

Note that $G=\Aut^0(X)\cong \IG^3_a \rtimes \IG_m$, and the minimal $G$-invariant subvarieties of $X$ are exactly the subvarieties of $H_u$. By $G$-equivariant weighted K-stability, it suffices to show $\delta^g_{p,\IT}(X)\ge 1$ for any closed point $p\in H_u$ (note that $\delta^g_{p,\IT}(X)\le 1$ by the choice of $\xi_0$). In the first step, we take refinement of $R_\bu$ by $H_u$ and denote it by $W_\bu$. In the second step, let $C=H_u\cap E_C$. There are three different cases of $p\in H_u$. Case (A) if $p\notin C$, we choose refinement by a line $l\seq H_u$ passing through $p$; case (B), if $p\in C$ and the tangent line $\{X=0\}$ of $C$ at $p$ has multiplicity two at $p$ when restricting to $C$, we choose refinement by the exceptional line $E$ of the $(2,1)$-weighted blowup of $H_u$ at the point $p$, where $\wt(X)=2, \wt(Y)=1$ and $(X,Y)$ is the local coordinate of $\IP^2$ at $p$; case (C), if $p\in C$ and the tangent line is of multiplicity three, we choose refinement by the exceptional line $E$ of the $(3,1)$-weighted blowup. Denote the second step refinement by $W^l_\bu$. Let $\widetilde{C}$ be the strict transform of $C$ in the above blowup. Finally, we conclude by the weighted Abban-Zhuang estimate that
$$\delta^g_{p,\IT}(X)\ge \min\{1, \delta^g_p(\IP^2; W_\bu)\},$$
where $\delta^g_p(\IP^2; W_\bu) \ge \frac{4}{3.773902}>1$ for any $p\in \IP^2 = H_u$, see Section \ref{Subsection: computing S^g of №2.28}. 

We next prove the weighted K-polystability of $(X,\xi_0)$. Let $Y\to X$ be the blowup of a point outside $H_u\cup E$. Then $Y$ admits a $\IT$-equivariant morphism $\tau:Y\to H_u$, which is the composition of the blow-down map of $E$ and the canonical map of $\IP^3$ with one point blowup to $H_u$. Hence $H_u$ can be viewed as a Chow quotient of $X$ by the $\IT$-action. Hence $\Val^{\IT,\circ}_X\cong \Val^{\circ}_{H_u}\times N_\IR $. Assume that $(X,\xi_0)$ is not weighted K-polystable. By \cite[Lemma 4.14]{BLXZ23}, there exist a $\IT$-invarant quasi-monimial valuation $v$ on $X$ which is not of the form $\wt_\xi$ such that $\delta^g(v)=1$. Up to twisting by some $\xi\in N_\IR$, we may assume that $v$ is the pull-back of some $v_0\in \Val^{\circ}_{H_u}$. 
For any point $q\in H_u$, there exists open neighbourhood $q\in U_0\seq H_u$ and $\IT$-invariant open neighbourhood $U\seq X$ such that $i: U\cong U_0\times \IA^1$ satisfying $\tau|_U = \pr_1 \circ i$. 
Let $q$ be the generic point of $C_{H_u}(v_0)$, then we have $$A_{H_u}(v_0) = A_{U_0}(v_0) = A_{U}(\tau^*v_0) = A_X(v). $$ 
On the other hand, the filtration induced by $v$ on $R_\bu$ and by $v_0$ on $W_\bu$ are the same. Hence $S^g(R_\bu;v)=S^g(W_\bu;v_0)$, and we get 
$$1=\frac{A_X(v)}{S^g(R_\bu;v)}=\frac{A_{H_u}(v_0)}{S^g(W_\bu;v_0)} \ge \delta^g_p(\IP^2; W_\bu)>1, $$
which is a contradiction. 
\end{proof}

\subsection{Computing $S^g$ of №2.28} 
\label{Subsection: computing S^g of №2.28}

Recall that $H_u$ is the strict transform of $\{u=0\}\seq \IP^3$ and $-K_X=\CO_{\IP^3}(4)-E_C$. We have decomposition $\CO_{\IP^3}(1)=H_u+E_C$ on $X$. Hence $-K_X-(w+1) H_u$ is ample for $w\le 0$, and has fixed part $w E_C$ and movable part $\CO_{\IP^3}(3-w)$ when $0< w < 3$. Since $\CO_{\IP^3}(1)|_{H_u}=\CO_{\IP^2}(1), E_C|_{H_u}=\CO_{\IP^2}(3)$, we have
\begin{eqnarray*}
W_{(1,w)} = 
\left\{ \begin{array}{ll}
H^0\Big(\IP^2, \CO(3+2w)\Big) 
& -1\le w < 0,\\
wC+H^0\Big(\IP^2, \CO(3-w)\Big) 
& 0\le w\le 3. 
\end{array} \right. 
\end{eqnarray*} 
This is the case of Example \ref{Example: formula when Q=P}. Now the question is reduced to the plane $H_u\cong \IP^2$. 

{\bf Case (A).} If $p\notin C$ and $l$ is a line on $\IP^2$. Then
\begin{eqnarray*}
\CF^{(y)}_l W_{(1,w)} = 
\left\{ \begin{array}{ll}
yl+H^0\Big(\IP^2, \CO(3+2w-y)\Big) 
& -1\le w < 0, 0\le y\le 3+2w\\
yl+wC+H^0\Big(\IP^2, \CO(3-w-y)\Big) 
& 0\le w\le 3, 0\le y\le 3-w. 
\end{array} \right. 
\end{eqnarray*} 
By (\ref{Formula: S^g 4}) we have 
\begin{eqnarray*} 
S^g(W_\bu; l) 
&=&\frac{1}{2!} 
\frac{1}{\Bv^g} 
\int_\IR
\int^3_{-1} g(w)
\vol(\CF^{(y)}W_{(1,w)}) 
 \dif w \dif y \\
&=& \frac{1}{2\Bv^g} 
\Big(
\int_{-1}^0 g(w) \int_0^{3+2w} (3+2w-y)^2 \dif y  \dif w 
+ \int_0^3 g(w) \int_0^{3-w} (3-w-y)^2 \dif y  \dif w 
\Big) \\
&\approx& 0.773902 \,\, <1, 
\end{eqnarray*}
where $g(w) = e^{-\xi_0 w}$ and $\Bv^g \approx 5.61542$. Hence $A_X(l)/S^g(W_\bu; l) >1$. 

Moreover, let $p_1,p_2,p_3\in C$ be the points that $l$ and $C$ intersect at (we just assume that $l$ intersects $C$ transversally). Then 
\begin{eqnarray*}
W^l_{(1,w,y)} = 
\left\{ \begin{array}{ll}
H^0\Big(\IP^1, \CO(3+2w-y)\Big) 
& -1\le w < 0, 0\le y\le 3+2w,\\
w(p_1+p_2+p_3)+H^0\Big(\IP^1, \CO(3-w-y)\Big) 
& 0\le w\le 3, 0\le y\le 3-w. 
\end{array} \right. 
\end{eqnarray*} 
Since $p\notin C$, we have 
\begin{eqnarray*}
\CF^{(x)}_p W^l_{(1,w,y)} = 
\left\{\begin{array}{ll}
xp+H^0\Big(\IP^1, \CO(3+2w-y-x)\Big) 
& (w,y,x)\in \D_-,\\
xp+w(p_1+p_2+p_3)+H^0\Big(\IP^1, \CO(3-w-y-x)\Big) 
& (w,y,x)\in \D_+. 
\end{array} \right. 
\end{eqnarray*} 
where $\D_-=\{-1\le w<0, 0\le y\le 3+2w, 0\le x\le 3+2w-y\}$ and $\D_+ = \{0\le w\le 3, 0\le y\le 3-w, 0\le x\le 3-w-y\}$. Hence 
\begin{eqnarray*} 
S^g(W^l_\bu; p) 
&=&\frac{1}{1!} 
\frac{1}{\Bv^g} 
\int^3_{-1} g(w)
\Big(
\int_{\D_-\cup \D_+}
\vol(\CF^{(x)}W_{1,w,y}) 
\dif x \dif y \Big)
 \dif w  \\
&=& \frac{1}{\Bv^g} 
\Big(
\int_{-1}^0 g(w) 
\int_0^{3+2w}
\int_0^{3+2w-y}
 (3+2w-y-x) \dif x  \dif y  \dif w  \\
&&\quad 
+\int_0^3 g(w) 
\int_0^{3-w} 
\int_0^{3-w-y}
(3-w-y-x) \dif x  \dif y  \dif w 
\Big) \\
&\approx& 0.773902 \,\, <1, 
\end{eqnarray*}
We conclude that $A_{\IP^1}(p)/S^g(W^l_\bu; p)>1$. 

%\begin{rmk} \rm We make a remark here that the two $S^g$ computed above are the same. One may show that $S^g(W_\bu;\CF_l)=\frac{1}{\Bv^g}\int_\BO y \LE^g=S^g(W^l_\bu, p)$. \end{rmk}

\begin{rmk} \rm
We remark here that the two $S^g$ computed above are the same. One may show that $S^g(W_\bu;\CF_l)= \int_\BO yg(w)\cdot\LE = S^g(W^l_\bu, p)$. 
\end{rmk}

{\bf Case (B).} If $p\in C$, and the tangent line is of multiplicity $2$ at $p$, and $E$ is the exceptional line of the $(2,1)$-weighted blowup $\tilde{\IP}^2\to \IP^2$. We denote by $L=\{X=0\}$ where $(X,Y)$ is a local coordinate of $\IP^2$ at $p$ such that $\ord_E(X)=2, \ord_E(Y)=1$. Let $P_0$ be the unique singular point of $\tilde{\IP}^2$, and $P_1=E\cap \tilde{L}, P_2=E\cap\tilde{C}$. The three points $P_0, P_1, P_2$ are different. For each $w, t$, the linear system $\CF^{(t)}_E W_{(1,w)}$ admits Zariski decomposition 
$$\CF^{(t)}_E W_{(1,w)} 
= 
N(\CF^{(t)}_E W_{(1,w)}) + 
H^0(\IP^2, P(\CF^{(t)}_E W_{(1,w)})). $$
For $-1\le w\le0$, we have 
\begin{eqnarray*}
N(\CF^{(t)}_E W_{(1,w)}) = 
\left\{ \begin{array}{ll}
tl
& 0\le t \le 3+2w, \\
tl+(t-3-2w)\tilde{L} 
& 3+2w\le t\le 6+4w, 
\end{array} \right. 
\end{eqnarray*}
and for $0\le w\le 3$,  
\begin{eqnarray*} 
N(\CF^{(t)}_E W_{(1,w)}) = 
\left\{ \begin{array}{ll}
2wl+ w\tilde{C}
& 0\le t \le 2w, \\
tl+w\tilde{C} 
& 2w\le t\le 3+w, \\
tl+w\tilde{C}+(t-3-w)\tilde{L} 
& 3+w\le t \le 6. 
\end{array} \right. 
\end{eqnarray*}
Then the positive part follows as 
\begin{eqnarray*}
P(\CF^{(t)}_E W_{(1,w)}) = 
\left\{ \begin{array}{ll}
\CO(3+2w) - N(\CF^{(t)}_E W_{(1,w)})
& -1\le w \le 0, \\
\CO(3-w)+w(2l+\tilde{C})-N(\CF^{(t)}_E W_{(1,w)})
& 0\le w\le 3. 
\end{array} \right. 
\end{eqnarray*}
The volume $\vol(\CF^{(t)}_E W_{(1,w)}) = P(\CF^{(t)}_E W_{(1,w)})^2$. Hence for $-1\le w\le 0$
\begin{eqnarray*}
\vol(\CF^{(t)}_E W_{(1,w)}) = 
\left\{\begin{array}{ll}
(3+2w)^2-\frac{1}{2}t^2
& 0\le t \le 3+2w, \\
(3+2w)^2-\frac{1}{2}t^2 + (t-3-2w)^2
& 3+2w\le t\le 6+4w, 
\end{array} \right. 
\end{eqnarray*}
and for $0\le w\le 3$, we have
\begin{eqnarray*}
\vol(\CF^{(t)}_E W_{(1,w)}) = 
\left\{ \begin{array}{ll}
(3-w)^2
& 0\le t \le 2w, \\
(3-w)^2 - \frac{1}{2}(t-2w)^2
& 2w\le t\le 3+w, \\
(3-w)^2 - \frac{1}{2}(t-2w)^2
+(t-3-w)^2
& 3+w\le t \le 6. 
\end{array} \right. 
\end{eqnarray*}
Hence we have 
\begin{eqnarray*} 
S^g(W_\bu; E) 
&=& \frac{1}{2\Bv^g} 
\Big(
\int_{-1}^0 g(w) 
\int_0^{6+4w} \vol(\CF^{(t)}_EW_{(1,w)}) 
\dif t 
\dif w  \\
&& \quad +  
\int_{0}^3 g(w) 
\int_0^{6} \vol(\CF^{(t)}_EW_{(1,w)}) 
\dif t 
\dif w  
\Big)\\
&\approx& 2.773902 \,\, < 3 \,\, 
=\,\, A_{\IP^2}(E), 
\end{eqnarray*}
that is, $A_{\IP^2}(E)/S^g(W_\bu; E)>1$. 

\begin{rmk} \rm
\label{Remark: 2.28 compute S^g using G}
There is another way to compute such $S^g$ using formula (\ref{Formula: S^g 2}). The concave transform of $\CF_E$ on $\BO$ is $G(w,y,x)=2x+y$, then $S^g(W_\bu;E)= \frac{1}{\Bv^g}\int_\BO (2x+y)\LE^g \approx 2.773902$. But we do not know how to get the further step refinement in this way. 
\end{rmk}

With the above calculation, taking the quotient of $\CF_E$ we get the second step refinement
\begin{eqnarray*}
W^E_{(1,w,t)} = 
\left\{ \begin{array}{ll}
H^0\Big(\IP^1, \CO(\frac{t}{2})\Big) 
& -1\le w < 0, 0\le t \le 3+2w, \\
(t-3-2w)P_1+H^0\Big(\IP^1, \CO(3+2w-\frac{t}{2})\Big) 
& -1\le w < 0, 3+2w\le t \le 6+4w, \\
wP_2 +H^0\Big(\IP^1, \CO(\frac{t}{2} -w)\Big) 
& 0\le w\le 3, 2w\le t\le 3+w, \\
wP_2 +(t-3-w)P_1 +H^0\Big(\IP^1, \CO(3-\frac{t}{2})\Big) 
& 0\le w\le 3, 3+w\le t\le 6. 
\end{array} \right. 
\end{eqnarray*}
Note that $\Diff_l(0) = \frac{1}{2} P_0$. Hence we have 
\begin{eqnarray*}
\delta^g(\IP^1;W_\bu^E)
=\min\Big\{\frac{1}{2 S^g(W^E_\bu;P_0)}, 
\frac{1}{S^g(W^E_\bu;P_1)}, 
\frac{1}{S^g(W^E_\bu;P_2)}\Big\}. 
\end{eqnarray*}
We first consider the filtration of $W^E_\bu$ induced by $P_0$. The fiberwise volume is clear in this case
\begin{eqnarray*}
\vol(\CF^s_{P_0}W^E_{(1,w,t)}) = 
\left\{ \begin{array}{ll}
\frac{t}{2}-s
& -1\le w < 0, 0\le t \le 3+2w,\\
3+2w-\frac{t}{2}-s 
& -1\le w < 0, 3+2w\le t \le 6+4w,\\
\frac{t}{2} -w-s
& 0\le w\le 3, 2w\le t\le 3+w,\\
3-\frac{t}{2}-s
& 0\le w\le 3, 3+w\le t\le 6. 
\end{array} \right. 
\end{eqnarray*}
Hence
\begin{eqnarray*} 
S^g(W^E_\bu; P_0) 
&=& \frac{1}{\Bv^g} 
\Big[
\int_{-1}^0 g(w) 
\Big(
\int_0^{3+2w}\frac{t^2}{8}\dif t+
\int_{3+2w}^{6+4w}\frac{1}{2}
(3+2w-\frac{t}{2})^2 \dif t
\Big)  \dif w  \\
&& \quad
+\int_0^3 g(w) 
\Big( 
\int_{2w}^{3+w} \frac{1}{2}
(\frac{t}{2}-w)^2 \dif t+
\int_{3+w}^6 \frac{1}{2}
(3-\frac{t}{2})^2 \dif t
\Big)  \dif w  
\Big]\\
&\approx& 0.386951 \,\, 
< \frac{1}{2} \,\, . 
\end{eqnarray*}
The filtrations induced by $P_1$ and $P_2$ are slightly different from that of $P_0$, we conclude that 
\begin{eqnarray*} 
S^g(W^E_\bu; P_1) -
S^g(W^E_\bu; P_0) 
&=& \frac{1}{\Bv^g} 
\Big[
\int_{-1}^0 g(w) 
\int_{3+2w}^{6+4w}
(3+2w-\frac{t}{2})(t-3-2w) \dif t
 \dif w  \\
&& \quad
+\int_0^3 g(w) 
\int_{3+w}^6 
(3-\frac{t}{2})(t-3-w) \dif t
 \dif w  
\Big] \,\,
\approx\,\, 0.386951, \\
S^g(W^E_\bu; P_2) -
S^g(W^E_\bu; P_0) 
&=& \frac{1}{\Bv^g} 
\int_0^3 g(w) w 
\Big(
\int_{2w}^{3+w}
(\frac{t}{2} - w) \dif t \\
&& \quad
+ \int_{3+w}^6 
(3-\frac{t}{2}) \dif t
\Big)
 \dif w  \,\,
\approx\,\, 0.226098. 
\end{eqnarray*}
Hence $\delta^g(\IP^1;W_\bu^E) \approx \frac{1}{0.773902}$. 

\begin{rmk}\rm 
We have $S^g(W^E_\bu; P_1)=2 S^g(W^E_\bu; P_0)$ and $S^g(W^E_\bu; P_2)+S^g(W^E_\bu; P_0)=1$. 
\end{rmk}

{\bf Case (C).} If $p\in C$ and $L=\{X=0\}$ tangent to $C$ of multiplicity $3$ at $p$, where $(X,Y)$ is a local coordinate of $\IP^2$ at $p$. Let $E$ be the exceptional curve of $(3,1)$-weighted blowup. With the same argument in the previous case or Remark \ref{Remark: 2.28 compute S^g using G}, we have $S^g(W_\bu;E)=\int_{\BO}(3x+y)\LE^g \approx 3.773902$. Hence $A_{\IP^2}(E)/S^g(W_\bu;E)\approx\frac{4}{3.773902}$, which is less than all the $A/S^g$ computed above. We will see that $\delta^g(X)$ is minimized by $E$. Firstly we have 
\begin{eqnarray*}
W^E_{(1,w,t)} = 
\left\{ \begin{array}{ll}
H^0\Big(\IP^1, \CO(\frac{t}{3})\Big) 
& -1\le w < 0, 0\le t \le 3+2w, \\
\frac{1}{2}(t-3-2w)P_1+
H^0\Big(\IP^1, \CO(\frac{1}{2}(3+2w-\frac{t}{3}))\Big) 
& -1\le w < 0, 3+2w\le t \le 9+6w, \\
wP_2 +H^0\Big(\IP^1, \CO(\frac{t}{3} - w)\Big) 
& 0\le w\le 3, 3w\le t\le 3+2w, \\
wP_2 +\frac{1}{2}(t-3-2w)P_1 +
H^0\Big(\IP^1, \CO(\frac{1}{2}(3-\frac{t}{3}))\Big) 
& 0\le w\le 3, 3+2w\le t\le 9. 
\end{array} \right. 
\end{eqnarray*}
Then for any point $P\ne P_1$ or $P_2$, we have 
\begin{eqnarray*} 
S^g(W^E_\bu; P) 
&=& \frac{1}{\Bv^g} 
\Big[
\int_{-1}^0 g(w) 
\Big(
\int_0^{3+2w}\frac{1}{2}(\frac{t}{3})^2\dif t+
\int_{3+2w}^{9+6w}\frac{1}{2}
(\frac{1}{2}(3+2w-\frac{t}{3}))^2 \dif t
\Big)  \dif w  \\
&& \quad
+\int_0^3 g(w) 
\Big( 
\int_{3w}^{3+2w} \frac{1}{2}
(\frac{t}{3}-w)^2 \dif t+
\int_{3+2w}^9 \frac{1}{2}
(\frac{1}{2}(3-\frac{t}{3}))^2 \dif t
\Big)  \dif w  
\Big] \\
&\approx& 0.257967, 
\end{eqnarray*} 
\begin{eqnarray*} 
S^g(W^E_\bu; P_1) -
S^g(W^E_\bu; P) 
&=& \frac{1}{\Bv^g} 
\Big[
\int_{-1}^0 g(w) 
\int_{3+2w}^{9+6w}
\frac{1}{2}(3+2w-\frac{t}{3})\cdot
\frac{1}{2}(t-3-2w)
\dif t
\dif w  \\
&& \quad
+\int_0^3 g(w) 
\int_{3+2w}^9 
\frac{1}{2}(3-\frac{t}{3}) \cdot
\frac{1}{2}(t-3-2w)
\dif t
\dif w  
\Big]\\
&\approx& 0.515935,\\
S^g(W^E_\bu; P_2) -
S^g(W^E_\bu; P) 
&=& \frac{1}{\Bv^g} 
\int_0^3 g(w) w 
\Big( 
\int_{3w}^{3+2w} 
(\frac{t}{3}-w) \dif t+
\int_{3+2w}^9 
\frac{1}{2}(3-\frac{t}{3}) \dif t
\Big)  \dif w  \\
&\approx& 0.226098. 
\end{eqnarray*}

\begin{rmk}\rm 
We have $S^g(W^E_\bu; P_1)=3 S^g(W^E_\bu; P)$ and $S^g(W^E_\bu; P_2)+2S^g(W^E_\bu; P)=1$. 
\end{rmk}

Note that $\Diff_l(0) = \frac{2}{3} P_0$, where $P_0$ is the unique singular point of the $(3,1)$-weighted blowup of $\IP^2$. Hence we have 
\begin{eqnarray*}
\delta^g(\IP^1;W_\bu^E)
=\min\Big\{\frac{1}{3 S^g(W^E_\bu;P_0)}, 
\frac{1}{S^g(W^E_\bu;P_1)}, 
\frac{1}{S^g(W^E_\bu;P_2)}\Big\}
\approx \frac{1}{0.773902}. 
\end{eqnarray*}

We conclude that $\delta^g(\IP^2; W_\bu) \approx \frac{4}{3.773902}$ is minimized by the exceptional line $E$ of the $(3,1)$-weighted blowup of $\IP^2$ at $p\in C$ where the tangent line has multiplicity $3$.

\subsection{Weighted K-polystability of №3.14} 
Any smooth Fano threefold $X$ in №3.14 is obtained by blowing up a Fano threefold $X_0$ in №2.28 at one point outside the plane containing $C$. We use the same notions as Section \ref{Subsection: №2.28}. We denote by $\hat{C}=\{y^2z-x(x-z)(x-\lam z)=0\}$ the cone over the curve $C$ in $\IP^3$, by $\hat{C}_{0,C}$ the strict transform of $\hat{C}$ via the two step blowups, and by $E_C, E_0$ the exceptional divisor of blowing up $C$ and the point $[1,0,0,0]$ respectively. 

We compute $R_\bu=R(X,-K_X)=R(X,\CO(4)-E_C-2E_0)$ on the affine open chart $X\setminus (\hat{C}_{0,C}\cup H_z)$. Then $u=w\cdot v, x=x_0\cdot z, y=y_0\cdot z$, and we have $E_0=\{z=0\}, E_C=\{v=0\}$. Hence 
\begin{eqnarray*}
H^0(X,-K_X) 
&=& H^0(X, -K_{X_0}-2E_0) \\
&=& \CF^2_{E_0}\big(z^3 H^0(\CO(1)) \oplus w H^0(\CO(3))\big)\cdot z^{-2} \\
&=& z^2 \la 1,x_0,y_0\ra \oplus wz\cdot \la 1,x_0,y_0\ra^3 \oplus w^2v\la 1,x_0,y_0\ra^2. 
\end{eqnarray*}
We take the faithful valuation by $\nv:w\mapsto (1,0,0), y_0\mapsto (0,1,0), x_0\mapsto (0,0,1)$, and we get a convex body by these sections. To make it is compatible with the moment polytope, we shift the first coordinate by one and denote the body by $\BO$, which is spanned by vertices $$(-1,0,0), (-1,0,1), (-1,1,0), (0,0,0), (0,0,3), (0,3,0), (1,0,1), (1,0,3), (1,2,1).$$ The volume of $\BO$ is $\frac{16}{3}$. Hence $\BO$ is a Okounkov body of $X$ since $(-K_X)^3=32$.  

Solving $\Fut_g(1)=0$, we get the soliton candidate numarically as $\xi_0 \approx 0.5265255550640977$. 

\begin{thm}
\label{Theorem: stability of 3.14}
Every smooth Fano threefold $X=X_{3.14}$ admits a K\"ahler-Ricci soliton. 
\end{thm}
\begin{proof}
We prove using the same argument as Theorem \ref{Theorem. stability of 2.28}. Note that the one step more blowup kills the $\IG_a$-part of the automorphism of $X_0$ and hence $G=\Aut^0(X)=\IG_m$. The minimal $G$-orbits of $X$ are exactly the closed subvarieties of $H_u$ and $E_0$. We shall take refinements by the toric divisors $H_u$ and $E_0$ on $X$ respectively. Let $W_\bu$ be the refinement of $R_\bu=R(-K_X)$ by $H_u$. For any point $p\in H_u$, we take $l$ be a line passing through $p$ if $p\ne C$, the $(k,1)$ blowup of $p$ if $p\in C$ and the tangent line $L$ of $C$ at $p$ has multiplicity $k$ at $p$. It suffices to show $\delta_l = \frac{A_{\IP^2}(l)}{S^g(W_\bu;l)}>1,\delta_p = \delta^g(\IP^1, \D_l; W^E_\bu) >1$, see Section \ref{Subsection: computing S^g of №3.14}. 

For points $p\in E_0$, the computation of $\delta_l$ and $\delta_p$ is totally the same as $p\in H_u$. The values of invariants are also the same. Hence the same argument of Theorem \ref{Theorem. stability of 2.28} shows that $X$ is weighted K-polystable. 
\end{proof}

\subsection{Computing $S^g$ of №3.14} 
\label{Subsection: computing S^g of №3.14}
Firstly we have
\begin{eqnarray*}
W_{(1,w)} = 
\left\{ \begin{array}{ll}
H^0\Big(\IP^2, \CO(3+2w)\Big) 
& -1\le w < 0,\\
wC+H^0\Big(\IP^2, \CO(3-w)\Big) 
& 0\le w\le 1. 
\end{array} \right. 
\end{eqnarray*} 
Note the formula of $S^g$ remains the same with the computation of №2.28 except we change the integration interval $0\le w \le 3$ into $0 \le w \le 1$. Here we list the results below:

{\bf (A).} If $p \notin C$, then we choose a general line $l$ on $H_u$ passing through $p$. Then 
$$S^g(W_{\bu};l) = S^g(W^{E}_{\bu};p) \approx 0.806338. $$  
Hence 
$\delta^g_p(\IP^2, W_\bu)\ge \min\{
\frac{1}{S^g(W_\bu; l)}, 
\frac{1}{S^g(W^E_\bu; p)} \} >1; $

{\bf (B).} If $p\in C$ and the tangent line is of multiplicity $2$, let $E$ be the exceptional line of the $(2,1)$-weighted blowup $\tilde{\IP}^2\to \IP^2$. Denote by $P_0$ the singular point of $\tilde{\IP}^2$ and $P_1=\tilde{L}\cap E, P_2=\tilde{C} \cap E$. We know that $P_0, P_1$, and $P_2$ are different points. Then
\begin{eqnarray*}
S^g(W_{\bu};E) \approx 2.806338, 
&&S^g(W^{E}_{\bu};P_0) \approx 0.403169, \\
S^g(W^{E}_{\bu};P_1) \approx 0.806338, 
&&S^g(W^{E}_{\bu};P_2) \approx 0.596831. 
\end{eqnarray*}
Hence $\delta^g_p(\IP^2, W_\bu)\ge \min\{
\frac{3}{S^g(W_\bu; E)}, 
\frac{1}{2S^g(W^E_\bu; P_0)}, 
\frac{1}{S^g(W^E_\bu; P_1)}, 
\frac{1}{S^g(W^E_\bu; P_2)}\} >1; $

{\bf (C).} If $p\in C$ and the tangent line is of multiplicity $3$, let $E$ be the exceptional line of the $(3,1)$-weighted blowup $\tilde{\IP}^2\to \IP^2$. Denote by $P_0$ the singular point of $\tilde{\IP}^2$ and $P_1=\tilde{L}\cap E, P_2=\tilde{C} \cap E$. The points $P_0, P_1, P_2$ are different. Then
\begin{eqnarray*}
S^g(W_{\bu};E) \approx 3.806338, 
&&S^g(W^{E}_{\bu};P_0) \approx 0.268799, \\
S^g(W^{E}_{\bu};P_1) \approx 0.806338, 
&&S^g(W^{E}_{\bu};P_2) \approx 0.462442. 
\end{eqnarray*}
Hence $\delta^g_p(\IP^2, W_\bu)\ge \min\{
\frac{4}{S^g(W_\bu; E)}, 
\frac{1}{3S^g(W^E_\bu; P_0)}, 
\frac{1}{S^g(W^E_\bu; P_1)}, 
\frac{1}{S^g(W^E_\bu; P_2)}\} >1.$

\section{GIT-stability and Weighted K-stability of Fano threefolds}
\label{Section: GIT and weighted K, 2.28 and 3.14}

In this section, we show that the weighted K-stability of $(X,\xi_0)$ is equivalent to the GIT-stability of plane cubic curves, where $X$ is a Fano threefold in the family №2.28 or №3.14. 

\begin{thm}
\label{Theorem: equivalence of weighted K-stability of №2.28 and №3.14 to the GIT-stability of plane cubic}
Let $X$ be the blowup of $\IP^3$ along a plane cubic curve $C\seq H\cong \IP^2$ or further blowing up a point outside $H$. Let $\xi_0$ be the soliton candidate of $X$. Then $(X, \xi_0)$ is weighted K-semistable (K-polystable) if and only if $C\seq H$ is GIT-semistable (GIT-stable or polystable). 
\end{thm}

The proof follows from the same calculation in the previous section, where the only difference is that we shall change the function $v$ of $C$ respectively. We will only do for №2.28. The computation for №3.14 is similar. 

It is well known that 
\begin{itemize}
    \item the plane cubic curve $C$ is GIT-stable if and only if it is smooth ($v=y^2z-x(x-z)(x-\lam z), \lam\ne 0,1$); 
    \item it is strictly GIT-semistable if and only if 
    it is a nodal cubic ($v=y^2z-x^2(x-z)$) or 
    the union of a conic and a secant line ($v=xyz-x^3$); 
    \item it is GIT-polystable if and only if it is the union of three lines ($v=xyz$) that do not intersect at one point; 
    \item it is GIT unstable for other cases, that is, 
    cuspidal cubic ($v=y^2z-x^3$) or 
    the union of a conic and a tangent line ($v=y^2z-x^2y$) or 
    the union of three lines with one common point ($v=y^3-x^2y$). 
\end{itemize}

In the first case, we have shown that $(X, \xi_0)$ is weighted K-polystable in the previous section. 

\begin{lem}
\label{Lemma: K-moduli semistable and unstable}
In the second and the third cases, $(X,\xi_0)$ is weighted K-semistable. 
In the fourth case, $(X,\xi_0)$ is weighted K-unstable. 
\end{lem}

\begin{proof}
We have the $\IN^2$-graded linear series on $H_u$ 
\begin{eqnarray*}
W_{(1,w)} = 
\left\{ \begin{array}{ll}
H^0\Big(\IP^2, \CO(3+2w)\Big) 
& -1\le w < 0,\\
wC+H^0\Big(\IP^2, \CO(3-w)\Big) 
& 0\le w\le 3. 
\end{array} \right. 
\end{eqnarray*} 

In the fourth case, let $l$ be the exceptional line of $(2,3)$-blowup of the point $\{x=y=0\}$. Then $\ord_l(v)\ge6$ and we have $S^g(W_\bu; l)\ge\frac{1}{\Bv^g}\int_\BO(2y+3y+6(x-y))\LE^g\approx 5.226098>5=A_{\IP^2}(l)$. Pulling-back $l$ to the cone $\tilde{l}$ over $l$ as a vertical divisor over $X$ we see that 
$$\frac{A_X(\tilde{l})}{S^g(R_\bu;\tilde{l})}
=\frac{A_{H_u}(l)}{S^g(W_\bu;l)} <1, $$
hence $(X,\xi_0)$ is weighted K-unstable. 

In the second and the third cases, $C$ has only an ordinary double point. Let $l$ be the exceptional line of the $(1,1)$-blowup of the ordinary double point. We have $\ord_l(v)=2, \ord_l(x)=\ord_l(y)=1$. Hence we have $S^g(W_\bu; l)=\frac{1}{\Bv^g}\int_\BO(y+y+2(x-y))\LE^g=2=A_{\IP^2}(l)$. Suppose that $\tilde{C}\cap l= \{P_1, P_2\}$, then the refinement of $W_\bu$ by $l$ is 
\begin{eqnarray*}
W^l_{(1,w,t)} = 
\left\{ \begin{array}{ll}
H^0\Big(\IP^1, \CO(t)\Big) 
& -1\le w < 0, 0\le t\le 3+2w,\\
w(P_1+P_2)+H^0\Big(\IP^1, \CO(t-2w)\Big) 
& 0\le w\le 3, 2w\le t\le 3+w. 
\end{array} \right. 
\end{eqnarray*} 
For $P\ne P_1, P_2$ we have $S^g(W^l_\bu; P) \approx 0.587831$ and $S^g(W^l_\bu; P_1) = S^g(W^l_\bu; P_2) \approx 0.625755$. 
We conclude by the weighted Abban-Zhuang estimate that $(X, \xi_0)$ is weighted K-semistable. 
\end{proof}

Next, we show that $(X,\xi_0)$ is strictly weighted K-semistable in the second case. Note that the nodal cubic $v=y^2z-x^2(x-z)$ is degenerated by $\lam(t)\cdot [u,x,y,z]=[u,t^{-1}x, t^{-1}y,z]$ to the union of three lines with no common point $v=z(y-x)(y+x)$. The union of a conic with a secant line $v=xyz-x^3$ is degenerated by $\lam(t)\cdot [u,x,y,z]=[u,t^{-1}x, y,z]$ to the union of three lines with no common point $v=xyz$. These degenerations induce the special degenerations of $(X,\xi_0)$ in the second case to the $(X, \xi_0)$ in the third case. It remains to show that $(X,\xi_0)$ is weighted K-polystable in the third case. 

\begin{lem}
\label{Lemma: K-moduli polystable}
Let $C$ be the union of three lines with no common point, and $X$ the blowup of $\IP^3$ along $C$. Then $(X,\xi_0)$ is weighted K-polystable. 
\end{lem}
\begin{proof}
%Let $X$ be the blowup of $\IP^3$ along $\{u=xyz=0\}$. Note that $X$ is toric in this case, and it admits three $\IG_m$-action: the horizontal one $\IT_h=[t^{-1}u,x,y,z]$ and the vertical ones $\IT_v=\IG_m^2$: $[u,t^{-1}x,y,z]$ and $[u, x, t^{-1}y,z]$. We denote by $\IT=\IT_h\times\IT_v$. One may compute directly that the $\BH$-minimizer $\xi_1\in N(\IT)_\IR$ (the soliton candidate of the toric variety $X$) indeed living in $N(\IT_h)_\IR$ and is just $\xi_0$. 
Note that $X$ is toric in this case. Since $(X,\xi_0)$ is weighted K-semistable, we conclude that $(X,\xi_0)$ is weighted K-polystable by \cite[Theorem 5.11]{BLXZ23}. 
\end{proof}

\begin{proof}[Proof of Theorem \ref{Theorem: equivalence of weighted K-stability of №2.28 and №3.14 to the GIT-stability of plane cubic}]
It follows directly from Lemma \ref{Lemma: K-moduli semistable and unstable} and Lemma \ref{Lemma: K-moduli polystable}. 
\end{proof}

\section{Weighted K-stability of cones and projective bundles}
\label{Section: Weighted K-stability of cones}

We establish the equivalence between the K-stability of a Fano manifold $V$ and the weighted K-stability of a projective cone $Y$ and bundle $\tilde{Y}$ over $V$ in this section, which generalizes the so-called Koiso's theorem \cite{Koi90}. We first recall some basic notions.

\subsection{Weighted K-stability of log Fano cones}

Let $(V, \D_V)$ be a $(n-1)$-dimensional log Fano pair such that $L=-\frac{1}{r}(K_V+\D_V)$ is an ample Cartier divisor for some $0<r\le n$. 
We define the projective cone over $V$ with polarization $L$ by 
$$Y=\overline{\CC}(V,L)\coloneqq\Proj\Big(\bigoplus_{m\ge0}\bigoplus_{0\le\lam\le m}H^0(V, (m-\lam)L)s^\lam\Big), $$
which is the union of the affine cone $\CC(V,L)=\Spec\big(\oplus_{m\ge0}H^0(V,  m L)\big)$ and a divisor $V_\infty =\{s=0\}$ at infinity. It admits a $\IG_m$-action along the cone direction, which we denote by $\IT_c$. Let $\D_Y$ be the closure of $\D_V\times \IC^*$ in $Y$. Then $-(K_{Y}+\D_Y)$ is Cartier and 
$$-(K_Y+\D_Y)\sim_\IQ(1+r)V_\infty. $$
Hence $(Y, \D_Y)$ is also log Fano, see for example \cite[Lemma 2.1]{ZZ22}. 
Let $\tilde{Y}\to Y$ be the blowup of the vertex, $V_0$ be the exceptional divisor and $\D_{\tilde{Y}}$ be the strict transform of $\D_Y$. Then $\tilde{Y}\cong \IP_V(\CO_V\oplus L)$ is a $\IP^1$-bundle over $V$. 
We have weight decomposition of $R=R(Y, \D_Y)$ via the $\IT_c$-action, 
$$H^0(Y, -m(K_Y+\D_Y)) = 
\bigoplus_{0\le\lam\le m(1+r)}H^0(V, (m(1+r)-\lam)L)s^\lam  $$
The toric divisor of the $\IT_c$-action is just $V_\infty$ and $V_0$. Since $A_{Y,\D_Y}(V_\infty)=1, A_{Y, \D_Y}(V_0)=r$, the moment polytope of this $\IT_c$-action is just $\BP=[-1, r]\seq M_\IR = \IR$ and the DH measure is $\DH_\BP(\dif \alpha) = \frac{(r-\alpha)^{n-1}L^{n-1}}{(n-1)!} \dif \alpha$. We set the soliton candidate $\xi_0$ to be the solution of $\Fut_g(1)=0$ where $g(\alpha)=e^{-\alpha\cdot\xi_0}$, that is
$$\int_{-1}^r \alpha \cdot g(\alpha) (r-\alpha)^{n-1} \dif \alpha = 0. $$

\begin{lem}
$\delta^{g}_{\IT_c}(Y,\D_Y) 
= \min\big\{1, 
\,\, 
\delta(V,\D_V) \big\}. $
\end{lem}
\begin{proof} 
Let $W_\bu$ be the refinement of $R_\bu=R(Y, \D_Y)$ by $V_\infty$. 
We claim that $\delta^{g}_{p}(V,\D_V;W_\bu)=\delta_{p}(V,\D_V)$ for any $p\in V_\infty$. It suffices to show that for any valuation $v$ on $V$, we have $S^g(W_\bu; v)=rS(L;v)$. Since $W_{(1,\alpha)}=(r-\alpha)L$ for any $-1\le \alpha \le r$, we have 
\begin{eqnarray*} 
S^g(W_\bu; v) 
&=& 
\frac{1}{(n-1)!\Bv^g} 
\int_{-1}^r g(\alpha) 
\int_{0}^{\infty} 
\vol(\CF_v^{(t)}(r-\alpha)L) 
\dif t
\dif \alpha \\
&=& 
\frac{1}{(n-1)!\Bv^g} 
\int_{-1}^r g(\alpha) (r-\alpha)^{n}
\dif \alpha 
\cdot
\int_{0}^{\infty} 
\vol(\CF_v^{(t)}L) 
\dif t. 
\end{eqnarray*} 
On the other hand 
$(n-1)!\Bv^g =
\int_{-1}^r g(\alpha) 
(r-\alpha)^{n-1} 
\dif \alpha 
\cdot 
\vol(L)$. Hence
\begin{eqnarray*} 
S^g(W_\bu; v) 
=
\frac{\int_{-1}^r g(\alpha) (r-\alpha)^{n}
\dif \alpha}
{\int_{-1}^r g(\alpha) (r-\alpha)^{n-1}
\dif \alpha}
\cdot
S(L;v). 
\end{eqnarray*} 
Since $(r-\alpha)^n= (r-\alpha)\cdot (r-\alpha)^{n-1}$ and 
$\int_{-1}^r \alpha\cdot g(\alpha) (r-\alpha)^{n-1}\dif \alpha=0$, we conclude that 
$$\int_{-1}^r g(\alpha) (r-\alpha)^{n}\dif \alpha
 =r\int_{-1}^r g(\alpha) (r-\alpha)^{n-1}\dif \alpha.$$ 

By Theorem \ref{Theorem: weighted AZ}, for any point $p\in V_\infty$, we have 
\begin{eqnarray*} 
\delta^{g}_{p, \IT_c}(Y,\D_Y;R_\bu) 
&\ge& \min\Big\{\frac{A_{Y,\D_Y}(V_\infty)}{
S^{g}(R_\bu; V_\infty)},
\,\, 
\delta^{g}_{p}(V,\D_V;W_\bu) \Big\}\\
&=& \min\Big\{1,
\,\, 
\delta_{p}(V,\D_V) \Big\}. 
\end{eqnarray*}
On the other hand, we should also consider the refinement $\tilde{W}_\bu$ of $R_\bu$ by $V_0$. One can show that $\tilde{W}_{(1,\lam)}=W_{(1,r-\lam)} (0\le \lam \le 1+r)$ on $V$, hence they have the same $S^g$ for any valuation on $V$. We conclude that 
$\delta^{g}_{\IT_c}(Y,\D_Y) 
\ge \min\big\{1, 
\,\, 
\delta(V,\D_V) \big\}. $

For the reverse inequality, it follows directly from 
$$\delta^g_{\IT_c}(Y,\D_Y)
\le \frac{A_{Y,\D_Y}(v)}{S^g(R_\bu;v)}
=\frac{A_{V,\D_V}(v_0)}{S^g(W_\bu;v_0)}
=\frac{A_{V,\D_V}(v_0)}{S(-K_V-\D_V;v_0)}, $$
where $v_0$ is a valuation on $V$ and $v$ is the pull-back of $v_0$ to $Y$. 
\end{proof} 

We have the following generalization of Koiso's theorem \cite{Koi90}. 

\begin{thm} 
\label{Theorem: soliton of cone}
The log Fano triple $(Y, \D_Y, \xi_0)$ is weighted K-semistable (weighted K-polystable) if and only if $(V,\D_V)$ is K-semistable (K-stable or K-polystable).
\end{thm}
\begin{proof}
The equivalence of semistability follows from the above lemma. For the equivalence of polystability, 
we prove with the same argument of Theorem \ref{Theorem. stability of 2.28}. Let $\IT_V = \IT(V,\D_V)$ be a maximal torus of $\Aut(V,\D_V)$. Then the $\IT_V$-action can be lifted to $(Y,\D_Y)$, and $\IT=\IT_c\times \IT_V$ is a maximal torus of $(Y, \D_Y)$. Then it also follows from 
\begin{eqnarray}
\frac{A_{Y,\D_Y}(v)}{S^g(R_\bu;v)}
=\frac{A_{V,\D_V}(v_0)}{S(-K_V-\D_V;v_0)}, \end{eqnarray}
where $v_0$ is a valuation on $V$ and $v$ is the pull-back of $v_0$ to $Y$. 
Indeed, by \cite{BLXZ23} the log Fano triple $(Y,\D_Y,\xi_0)$ is strictly weighted K-semistable if and only if there exists a valuation $v\ne\wt_\xi, \xi\in N(\IT)_\IR$ such that $A_{Y,\D_Y}(v)=S^g(R_\bu;v)$. By the above equality, this holds if and only if there exists a valuation $v_0$ on $V$ whose pull-back to $Y$ is $v$, such that $v_0\ne \wt_\xi, \xi\in N(\IT_V)_\IR$ and $A_{V,\D_V}(v_0)=S(-K_V-\D_V;v_0)$. This is equivalent that $(V,\D_V)$ is strictly K-semistable by \cite{LXZ22}. 
\end{proof}

The theorem will lead to new examples of weighted K-moduli spaces based on the existence of K-moduli spaces of log Fano pairs \cite{LXZ22}. In particular, it gives us examples of strictly weighted K-semistable Fano varieties. Let $\pi: (\CV, \CL) \to \IA^1$ be a TC of $(V, L)$. We can also define the cone over $(\CV, \CL)$ 
$$\CY 
= \overline{\CC}(\CV,\CL) 
\coloneqq\Proj_{\IA^1}
\Big(
\bigoplus_{m\ge0}\bigoplus_{0\le\lam\le m}
\pi_*(\CL^{\otimes(m-\lam)}) \cdot s^\lam
\Big), $$
which gives a TC of $Y$ whose central fiber is the cone $\CY_0=\overline{\CC}(\CV_0, \CL_0)$. The $\IT_c$-action on $Y$ extends to the whole family $\CY$ and is compatible with the $\IT_c$-action on the central fiber. Hence we have 

\begin{cor} 
Let $V$ be a strictly K-semistable Fano manifold with $-K_V \sim_\IQ r L$ for some ample Cartier divisor $L$ and $0<r \le 1$. Then the cone $(Y=\overline{\CC}(V; L), \xi_0)$ is strictly weighted K-semistable. 
\end{cor} 

\begin{ex}\rm
It was shown in \cite{Tia97} that there exists a Fano manifold $V$ in the family №1.10 of Mori and Mukai's list such that there exists a non-product type special degeneration $(\CV, \CL)$ of $(V, L)$ whose central fiber $(\CV_0, \CL_0)$ is the Mukai-Umemura manifold. Since $\CV_0$ is K-polystable, we know that $V$ is strictly K-semistable. The above Corollary gives us a four-dimensional Fano variety $Y=\overline{\CC}(V,L)$ with the vertex as the only singular point, which has a special degeneration to the cone $\CY_0=\overline{\CC}(\CV_0,\CL_0)$ over the Mukai-Umemura manifold. By Theorem \ref{Theorem: soliton of cone} we see that $(\CY_0, \xi_0)$ is weighted K-polystable. Hence $(Y,\xi_0)$ is strictly weighted K-semistable.  
\end{ex}

\subsection{Weighted K-stability of projective bundles over log Fano pairs}
Let $r\in\IQ_{>0}, (V, \D_V)$ and $(\tilde{Y}, \D_{\tilde{Y}})$ be the same as above. The pair $(\tilde{Y}, \D_{\tilde{Y}})$ is log Fano if $r>1$. For $0<r\le 1$, the anti-canonical divisor of $(\tilde{Y}, \D_{\tilde{Y}})$ is not ample. Instead, we consider the log Fano pair $(\tilde{Y}, \D_{\tilde{Y}}+aV_0)$ where $1-r<a<1$. The $\IT_c$-action of $(Y,\D_Y)$ lifts to $(\tilde{Y}, \D_{\tilde{Y}}+V_0)$ naturally. 

If $r>1$, then $A_{\tilde{Y}, \D_{\tilde{Y}}}(V_0)=A_{\tilde{Y}, \D_{\tilde{Y}}}(V_\infty)=1$. We have weight decomposition 
$$
H^0(\tilde{Y}, 
-m(K_{\tilde{Y}}+\D_{\tilde{Y}})) 
= 
\bigoplus_{-1\le\alpha\le 1, m\alpha \in \IN} 
H^0(V, m(r-\alpha)L)s^{m(1+\alpha)}. $$
If $r\le 1$, then $A_{\tilde{Y}, \D_{\tilde{Y}}+aV_0}(V_0)=1-a$ and $A_{\tilde{Y}, \D_{\tilde{Y}}+aV_0}(V_\infty)=1$. We also have 
$$
H^0(\tilde{Y}, 
-m(K_{\tilde{Y}}+\D_{\tilde{Y}}+aV_0)) 
= 
\bigoplus_{-1\le\alpha\le 1-a, m\alpha \in \IN} 
H^0(V, m(r-\alpha)L)s^{m(1+\alpha)}. $$
Hence the moment polytope $\BP=[-1,1]$ for $r>1$ and $\BP=[-1,1-a]$ for $0<r\le 1$, and the DH measure is $\DH_\BP(\dif \alpha) = \frac{(r-\alpha)^{n-1}L^{n-1}}{(n-1)!} \dif \alpha$. Solving equation in $\xi_0$ 
$$\int_\BP \alpha \cdot g(\alpha) (r-\alpha)^{n-1} \dif \alpha = 0, \quad
g(\alpha)=e^{-\alpha\cdot\xi_0}, $$
we get the soliton candidate $\xi_0\in\IR$. 
Following the same proof of Theorem \ref{Theorem: soliton of cone}, we also have 
\begin{thm} 
The log Fano triples $(\tilde{Y}, \D_{\tilde{Y}}, \xi_0)$ for $r>1$ and $(\tilde{Y}, \D_{\tilde{Y}}+aV_0, \xi_0)$ for $0<r\le 1$ are weighted K-semistable (weighted K-polystable) if and only if $(V,\D_V)$ is K-semistable (K-stable or K-polystable).
\end{thm}

\section{Conflict of interest statement}

On behalf of all authors, the corresponding author states that there is no conflict of interest.

\section{Data availability statement}

Our manuscript has no associated data.

\appendix
\section{Proof of Lemma \ref{Lemma: uniform bound of S^g_m/S^g}}
It follows directly from the same argument of \cite[Corollary 2.10]{BJ20} based on the following Lemma, which is a minor strengthening of \cite[Lemma 2.2]{BJ20}. 
\begin{lem}
For any $\varepsilon>0$ there exists $m_0=m_0(\varepsilon)$ such that $\int_\BO G \cdot \LE^g_m \le \int_\BO G\cdot\LE^g + \varepsilon$ for any $m\ge m_0$ and every concave function $G$ on $\BO$ satisfying $0\le G\le 1$. 
\end{lem}

\begin{proof}
For any $t>0$, let $\BO^t=\{\alpha\in\IR^n\mid\alpha+[-t,t]\seq \BO\}$. This is a closed subset of $\BO$ and converges to $\BO$ as $t\to 0$. 

Fix $\varepsilon>0$. There exists $t>0$ such that $\LE(\BO\setminus\BO^{2t})<\varepsilon/4$. Since $\LE_m\to \LE$ converges weakly, we have $\limsup\, \LE_m(\BO\setminus\BO^{t})\le \LE(\BO\setminus\BO^{2t}) \le \varepsilon/4$. So there exists $m_1$ such that $\LE_m(\BO\setminus\BO^{t})\le \varepsilon/2$ for any $m\ge m_1$. Note that $g\circ p$ is uniformly continuous on $\BO$. Hence there exists $m_2$ such that $|g(p(\alpha))-g(p(\beta))| < \varepsilon/(2\LE(\BO))$ for any $\alpha, \beta\in\BO$ with $|\alpha-\beta|\le 2/m_2$. We may choose $m_0$ such that $m_0 \ge m_1,m_2$ and $t^{-1}$. 

For any $m\ge m_0$, we set 
\begin{eqnarray*} 
A'_m
&\coloneqq& 
\{\alpha \in \frac{1}{m}\IZ^n\mid
\alpha + [0,\frac{1}{m}]^n \seq \BO\}, \\
A_m
&\coloneqq&  
\{\alpha \in \frac{1}{m}\IZ^n\mid
\alpha + [-\frac{1}{m},\frac{1}{m}]^n \seq \BO\}.
\end{eqnarray*} 
It's clear that $\BO^t\cap \frac{1}{m}\IZ^n \seq A_m \seq A'_m$. 
For any $\alpha\in A'_m$, we denote by $g^\alpha_\max, g^\alpha_\min$ the maximum and minimum value of $g\circ p$ on $\alpha+[0,\frac{1}{m}]^n$ respectively. Since $m\ge m_2$, we have $g^\alpha_\max-g^\alpha_\min \le \varepsilon/(2\LE(\BO))$. We define
\begin{eqnarray*} 
f(\alpha)
&=&
\Big(
\sum_{w\in \{0,\frac{1}{m}\}^n} 
2^{-n}(G\cdot g\circ p)(\alpha +w)
\Big) 
- \int_{[0,\frac{1}{m}]^n} 
(G\cdot g\circ p)(\alpha +w)  \dif w . 
\end{eqnarray*}
Then we have
\begin{eqnarray*} 
f(\alpha)
&\le& g^\alpha_\max \Big(
\sum_{w\in \{0,\frac{1}{m}\}^n} 
2^{-n}G(\alpha +w)
- \int_{[0,\frac{1}{m}]^n} 
G(\alpha +w)  \dif w 
\Big)  \\
&& + (g^\alpha_\max-g^\alpha_\min) 
\int_{[0,\frac{1}{m}]^n} 
G(\alpha +w)  \dif w . 
\end{eqnarray*}
The first term on the right-hand side of the inequality is non-positive by the concavity of $G$. Hence
\begin{eqnarray*}
f(\alpha) \le 
\frac{\varepsilon}{2\LE(\BO)}\cdot 
\int_{\alpha + [0,\frac{1}{m}]^n} 
G\cdot \LE. 
\end{eqnarray*}
Then $\sum_{\alpha\in A'_m}
f(\alpha) \le \varepsilon/2$. We are ready to give the desired estimate. 
\begin{eqnarray*} 
&&\int_\BO G\cdot\LE^g 
\,\,=\,\,
\int_\BO G\cdot g\circ p\cdot \LE\\
&\ge&
\sum_{\alpha\in A'_m} 
\int_{\alpha + [0,\frac{1}{m}]^n} 
G\cdot g\circ p\cdot \LE 
\,\,=\,\, \sum_{\alpha\in A'_m}
\int_{[0,\frac{1}{m}]^n} 
(G\cdot g\circ p)(\alpha +w)  \dif w  \\
&=& \sum_{\alpha\in A'_m}
\Big(
\sum_{w\in \{0,\frac{1}{m}\}^n} 
2^{-n}(G\cdot g\circ p)(\alpha +w) 
-f(\alpha)
\Big) 
\,\,\ge\,\, \sum_{\alpha\in A_m}
(G\cdot g\circ p)(\alpha ) 
-\Big(
\sum_{\alpha\in A'_m}
f(\alpha)
\Big) \\
&\ge&
\int_{\BO^t}
(G\cdot g\circ p)\cdot \LE_m
-\frac{\varepsilon}{2} 
\,\,\ge\,\, 
\int_{\BO}
(G\cdot g\circ p)\cdot \LE_m
-\LE_m(\BO\setminus \BO^t)
-\frac{\varepsilon}{2} 
\,\,\ge\,\, 
\int_{\BO}
G\cdot \LE^g_m
-\varepsilon. 
\end{eqnarray*}
\end{proof}

\bibliographystyle{alpha}
\bibliography{ref}

\end{document}